%% file: main.tex
\documentclass[a4paper]{article}

\pdfoutput=1

\input{settings}
\input{tikz_commands}

\author{Mark C. Bell \and Saul Schleimer}

\title{The word problem for the mapping class group \\
in quasi-linear time}

\begin{document}

\maketitle

\begin{abstract}
We give an $O(n \log^3(n))$--time algorithm for the word problem in the mapping class group of a compact surface.
\blfootnote{This work is in the public domain.}
\end{abstract}


\ccode{20F10, 57K20, 68W40}


\section{Introduction}

Suppose that $S$ is a compact surface, possibly with boundary.
We define $\MCG(S)$, the \emph{mapping class group} of $S$, to be the group of isotopy classes of homeomorphisms of $S$.
Since this group is finitely generated~\cite{Dehn38} one may ask for a solution to its \emph{word problem}: 
an algorithm that, given a word of length $n$ in the generators and their inverses, decides whether the resulting mapping class is trivial. 

Our main result is the following.

\begin{theorem}
\label[theorem]{WordProblem}
There is an algorithm that solves the word problem for $\MCG(S)$ in $O(n \log^3(n))$ time.
Furthermore, the implied constants are bounded by a fixed polynomial in $|\chi(S)|$.
\end{theorem}

This answers a question of Farb~\cite[Question~3.1]{Farb06}; 
it also gives the best solution to date to Thurston's Problem~20~\cite[Page~380]{Thurston82}.
We prove \cref{WordProblem} using divide-and-conquer, ideas of Dynnikov~\cite{Dynnikov22}, and the following. 

\begin{theorem}
\label[theorem]{GeometricIntersection}
There is an algorithm to compute the geometric intersection number of two curves (given in $\Delta$--coordinates) in $O(n \log^2(n))$ time.
Furthermore, the implied constants are bounded by a fixed polynomial in $|\chi(S)|$.
\end{theorem}

Our proof of \cref{GeometricIntersection} is inspired by the \emph{half-GCD algorithm}~\cite{Moller08};  
this uses divide-and-conquer to accelerate the euclidean algorithm.
As a special case of \cref{GeometricIntersection}, we also obtain  the best solution to date for a question (posed as a ``remarque'') of~\cite[Expos\'{e}~4, page~66]{FLP12}.

\begin{theorem}
\label[theorem]{CurveShortening}
There is an algorithm for \emph{curve shortening} in $S$ which runs in $O(n \log^2(n))$ time.
Furthermore, the implied constants are bounded by a fixed polynomial in $|\chi(S)|$. \qed
\end{theorem}

The final divide-and-conquer, giving the last factor of $\log(n)$, comes from the recent improvement in large integer multiplication~\cite[Theorem~1.1]{HarveyHoeven21}. 
Note, however, that our paper uses integer multiplication as a ``black box''. 
Swapping in a different algorithm for integer multiplication, say running in time $M(n)$, changes the running time in \cref{WordProblem} to $O(M(n) \log^2(n))$ and in Theorems 
\ref{GeometricIntersection} and \ref{CurveShortening} to $O(M(n) \log(n))$.

\subsection{Reduction to a special case}

The general case of \cref{WordProblem} reduces, in linear time, to the following special case:
\begin{itemize}
\item
The surface $S$ is connected and oriented,
\item
the surface $S$ has $\chi(S) \leq -2$,
\item
the mapping class $f$ preserves orientation, and
\item
the mapping class $f$ preserves the boundary of $S$, componentwise.
\end{itemize}

\noindent
We henceforth assume that $S = S_{g, b}$: the surface of genus $g$ with $b$ boundary components.
We use the notation $\Mod(S) \defeq \MCG^+_\bdy(S)$ for the group of such mapping classes.
It is a theorem from Dehn's Breslau lectures~\cite{Dehn87} (see also \cite{Dehn38, Lickorish64, Humphries79}) that $\Mod(S)$ is finitely generated by Dehn twists. 
Fix $F$, a finite generating set of $\Mod(S)$.

\subsection{Corollaries}

\begin{remark}
There is another common variant of the mapping class group, $\MCG^+(S, \bdy S)$, where homeomorphisms and isotopies fix $\bdy S$ pointwise.  
The word problem in this case is reduced to the previous one by attaching pairs of pants to all boundary components of $S$.
\end{remark}

\begin{remark}
Thus for any fixed value of $r$ we obtain an $O(n \log^3(n))$ algorithm to solve the word problem in $B_r$:
the \emph{braid group} on $r$ strands.  
This answers Problem~B13 on the list maintained by Shpilrain~\cite{Shpilrain24}.
\end{remark}

\begin{remark} 
The order of torsion in $\MCG(S)$ is bounded, solely in terms of the topology of $S$; 
this follows from the classification of hyperbolic orbifolds and a result of Nielsen~\cite[page~24, Staz]{Nielsen43}.
Thus we have an $O(n \log^3(n))$ algorithm to decide whether a given mapping class has finite order.
\end{remark}

\subsection{History and other work}

In a pair of papers~\cite{Dehn11, Dehn38}, and in his Breslau lectures~\cite{Dehn87}, Dehn 
\begin{itemize}
\item 
sets out the word problem for finitely generated groups, 
\item 
defines the mapping class group $\MCG(S)$,
\item
proves that the mapping class group is finitely generated\footnote{Dehn assumes that $S$ is oriented. 
For non-orientable $S$, finite generation is due to Chillingworth~\cite{Chillingworth69}.}, and 
\item 
gives two solutions to its word problem.
\end{itemize}
Dehn's first solution relies on the action of the mapping class group on $\pi_1(S)$.
Since Dehn's work predates the invention of computational complexity by many decades~\cite{FortnowHomer03}, he gives no time estimates.
However, a \naive{} analysis shows that his first algorithm is exponential time.
It is accelerated to polynomial time by~\cite[Theorem~A.8]{Schleimer08}.

Dehn's second solution relies on what he calls the \emph{arithmetic field} $\calC(S)$:
(essentially) the set of isotopy classes of simple closed multi-curves in $S$.  
Dehn shows that a choice of \emph{pants decomposition} for $S$ equips $\calC(S)$ with \emph{intersection/twist} coordinates.
Furthermore, the natural action of $\MCG(S)$ on $\calC(S)$ is piecewise linear in these coordinates.
A \naive{} analysis shows that his second algorithm is quadratic time.

Parts of Dehn's work were recovered, and parts greatly extended,
by Lickorish (1960s) and Thurston (1970s). 
For a discussion, see Stillwell's translator's notes in~\cite{Dehn87}.
Other quadratic-time algorithms for the word problem, since Dehn's, include
\cite{Mosher99, Takarajima00a, Takarajima00b, Hamidi-Tehrani00, Thurston08}.
Most recently Dynnikov~\cite{Dynnikov22} has given a quadratic-time algorithm using \emph{curve shortening}.  
His paper is one of the inspirations for our work; 
we refer to the appropriate sections of~\cite{Dynnikov22} as they arise. 
We refer to Erickson--Nayyeri~\cite{EricksonNayyeri13} for a history of curve shortening, an extensive bibliography, and state-of-the-art algorithms in the RAM model.




Programs that solve the word problem (often as just a small part of their functionality) include 
\cite[Appendix~C, Twist]{Penner82}, 
\cite[BH]{BH}, 
\cite[XTrain]{XTrain}, 
\cite[Trains]{Trains}, 
\cite[Dynn]{Dynn}, 
\cite[Flipper]{flipper}, and 
\cite[Curver]{curver}. 
Other relevant programs include
\cite[Branched]{branched} and 
\cite[Teruaki]{teruaki}.

\subsection*{Acknowledgements}

We thank Rich Schwartz, Sam Taylor, and Richard Webb for enlightening conversations. 
We thank Greg Kuperberg and Filippo Baroni for insightful comments on an earlier draft.

The first author thanks Gordon, Rachael, Caroline, and Katie for their steadfast personal support throughout the course of this work.

\section{Background}

We use multi-tape Turing machines~\cite[page 161]{HopcroftUllman79} for our model of computation.
We define the \emph{complexity} of:
\begin{itemize}
\item an integer $n$ to be $\complexity{n} \defeq \lceil \log(|n| + 1) \rceil$,
\item a vector $v = (v_i)$ to be $\complexity{v} \defeq \sum_i \complexity{v_i}$, and
\item a matrix $M = (m_{i,j})$ to be $\complexity{M} \defeq \sum_{i,j} \complexity{m_{i, j}}$.
\end{itemize}
These measure the number of bits needed to write down each object.
Then if $a$ and $b$ are integers such that $n \defeq \max(\complexity{a}, \complexity{b})$ we have that:
\begin{itemize}
\item $a \pm b$ can be computed in $O(n)$ time and
\item $a \cdot b$, $a \fdiv b \defeq \lfloor a / b \rfloor$ and $a \cdiv b \defeq \lceil a / b \rceil$ can all be computed in $O(n \log(n))$ time \cite[Theorem~1.1]{HarveyHoeven21}.
\end{itemize}
All of our algorithms are recursive with only a constant number of variables (bounded by a fixed polynomial in $|\chi(S)|$) in each frame.
Therefore at each stage all variables that we need to read or write can be reached within $O(n)$ time.
Since this overhead is smaller than all operations that we will perform, we do not track accessing variables.

\subsection{Matrices and further reduction}

Suppose that $d$ is a non-negative integer. 
We treat $\GL(d, \ZZ)$ as both a motivating example and as a useful tool. 
Fast integer multiplication gives fast matrix multiplication, as follows.

\begin{corollary}
\label[corollary]{two_matrix_multiply}
Suppose that $A$ and $B$ are matrices in $\GL(d, \ZZ)$ with $\complexity{A}, \complexity{B} \leq n$.  
Then we can compute the product $A \cdot B$ in $O(n \log(n))$ time.
Furthermore, the implied constants are bounded by a fixed polynomial in $d$. \qed
\end{corollary}

From \cite{Agol21}, and \cite[Proposition~2]{OlshanskiiShpilrain24}, we have the following. 

\begin{proposition}
\label[proposition]{many_matrix_multiply}
Suppose that $(E_n, \ldots, E_2, E_1)$ is a list of elementary matrices in $\GL(d, \ZZ)$. 
Then we can compute their product in $O(n \log^2(n))$ time.
Furthermore, the implied constants are bounded by a fixed polynomial in $d$.
\end{proposition}

\begin{proof}
We split the list $(E_i)$ in half and recurse.
The two subproducts have complexity at most $O(n)$, by induction.
Thus we can compute the final product in time $O(n \log(n))$ by \cref{two_matrix_multiply}.

Let $t(n)$ bound the time needed to compute the product of $n$ elementary matrices.
We deduce that
\[
t(n) \leq 2 t(n/2) + O(n \log(n))
\]
Thus $t(n) = O(n \log^2(n))$ by the master theorem~\cite[Theorem~4.1]{CLRS22}.
\end{proof}

\begin{corollary}
The word problem for $\GL(d, \ZZ)$, as generated by elementary matrices, can be solved in $O(n \log^2(n))$ time. 
Furthermore, the implied constants are bounded by a fixed polynomial in $d$. 
\qed
\end{corollary}

Recall that $F$ is our fixed finite generating set for $\Mod(S)$.

\begin{corollary}
Suppose that $f_n, \ldots, f_2, f_1 \in F \cup F^{-1}$ is a list of generators and inverses of generators.
Then the action of $f \defeq f_n \circ \cdots \circ f_2 \circ f_1$ on $H_1(S)$ (given as an integer matrix) can be computed in $O(n \log^2(n))$ time. \qed
\end{corollary}

Using this, from now on we may assume that certain mapping classes act trivially on $H_1(S)$.

\section{Curves, pants decompositions, and coordinates}

Let $\calC(S)$ denote the set of isotopy classes of essential non-peripheral simple closed multi-curves in $S$.  
As usual and when needed, we will blur the distinction between a multi-curve and its isotopy class.
For $\alpha$ and $\beta$ in $C(S)$ we use $\iota(\alpha, \beta)$ to denote their \emph{geometric intersection number}: the minimal intersection number among representatives of $\alpha$ and $\beta$.
We use $\calC_0(S) \subset \calC(S)$ for the subset of curves (multi-curves with one component). 

\subsection{Cuffs, dual curves, and double duals}

In our construction of pants decompositions, dual curves, and double duals we closely follow~\cite[Expos\'{e}~4, Section~III]{FLP12}.

Fix a \emph{pants decomposition} $P \subset \calC_0(S)$ of $S$: 
a maximal collection of disjoint non-parallel curves.
Note that $|P| = 3g - 3 + b$ where $g$ is the genus of $S$ and $b$ is its number of boundary components.
The curves $\alpha \in P$ are the \emph{cuffs} of the pants decomposition.
Each component of $S - P$ is a \emph{pair of pants}: 
its completion (in the induced path metric) is a copy of $S_{0, 3}$.
Note that $|S - P| = 2g - 2 + b$.

For each cuff $\alpha \in P$ define $X_\alpha$ to be the non-pants component of $S - (P - \alpha)$.  
Thus $X_\alpha$ is either a once-holed torus or a four-holed sphere.
Following~\cite[page~62]{FLP12}, we arrange matters so that all of the $X_\alpha$ are four-holed spheres.

A curve $\beta \in \calC_0(S)$ is \emph{dual to $\alpha$ relative to $P$} if $\beta$ lies in $X_\alpha$ and satisfies $\iota(\alpha, \beta) = 2$. 
We fix $Q \subset \calC_0(S)$ to be a collection of dual curves relative to $P$, one per cuff.

Lastly, for each cuff $\alpha \in P$ and its dual $\beta \in Q$, their \emph{double dual} is $\gamma \defeq T_\alpha(\beta)$; the image of $\beta$ under the right Dehn twist about $\alpha$.
We fix $R \subset \calC_0(S)$ to be the set of \emph{double dual} curves.

For an example of this, see \cref{cuff_dual_double}.

\begin{figure}[ht]
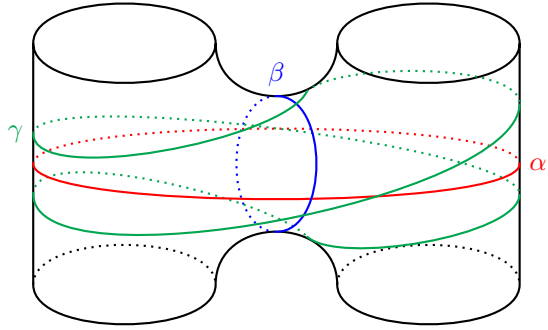

\centering
\include{tikz/example_S_0_4}
\caption{The cuff $\alpha$, its dual $\beta$, and their double dual $\gamma$.}
\label[figure]{cuff_dual_double}
\end{figure}

From now on, we fix $\Delta = P \cup Q \cup R$.
So $\Delta$ is a collection of $|\Delta| = 3(3g - 3 + b)$ curves.

\begin{lemma}[{\cite[Proposition~5.3]{Dynnikov22}, \cite[Proposition~2.8]{FarbMargalit12}}]
\label[lemma]{DeltaSuffices}
Suppose that $f$ is a mapping class acting trivially on $H_1(S)$. 
Then $f$ is the identity if and only if $f(\delta) = \delta$ for each $\delta \in \Delta$. \qed
\end{lemma}

\begin{remark}
We only need \cref{DeltaSuffices} to deal with hyperelliptic elements. 
If we avoid surfaces $S$ admitting such (namely $S_{2, 0}$, $S_{1,2}$, $S_{1, 1}$, $S_{1, 0}$, and $S_{0, 4}$) then a mapping class $f \in \Mod(S)$ is the identity if and only if $f(\delta) = \delta$ for each $\delta \in \Delta$.
\end{remark}

\subsection{Curve coordinates}
\label{Sec:Coordinates}

\begin{definition}
Suppose that $\epsilon \in \calC(S)$ is a multi-curve.
We define the \emph{$\Delta$--coordinates} of $\epsilon$ to be the following vector.
\[
\Delta(\epsilon) = (\iota(\delta, \epsilon))_{\delta \in \Delta} \in \ZZ^\Delta \qedhere
\]
\end{definition}

We also need the \emph{$(m,s,t)$--coordinates} from Expos\'{e}~4 and Expos\'{e}~6 of~\cite{FLP12}.
These require a choice of pants decomposition (and dual and double dual curves) -- we again use $\Delta$.  
Note that~\cite[Expos\'{e}~4, Th\'{e}or\`{e}me~7]{FLP12} states that $(m, s, t)$--coordinates are an injection of $\calC(S)$ into $\ZZ^{3(3g - 3 + b)}$, justifying their name.
This injection is made explicit in the appendix to Expos\'{e}~6~\cite[page~120]{FLP12}.

\begin{proposition}
\label[proposition]{ConvertDeltaMST}
There is a linear-time algorithm to convert from $\Delta$--coordinates to $(m, s, t)$--coordinates. \qed
\end{proposition}

It follows that $\Delta$--coordinates are also injective, justifying their name.

\subsection{Dynnikov's matrices for mapping classes}

Suppose that $f \in \Mod(S)$ is a mapping class. 

\begin{definition}
Following Dynnikov~\cite[Section~6]{Dynnikov22}, we define the \emph{$\Delta$--coordinates} of $f$ to be the following matrix.
\[
\Delta(f) = (\Delta(f(\epsilon)))_{\epsilon \in \Delta}
\]
We adopt the convention that the $(\delta, \epsilon)$--entry of $\Delta(f)$ is $\intersection(\delta, f(\epsilon))$.
\end{definition}

\begin{remark}
\label[remark]{symmetry_of_delta}
Applying the symmetry and invariance of geometric intersection, we have
\[
\intersection(\delta, f(\epsilon)) = \intersection(f(\epsilon), \delta) = \intersection(\epsilon, f^{-1}(\delta)).
\]
Thus $\Delta(f^{-1})$ is the transpose of $\Delta(f)$.
\end{remark}

\begin{example}
Suppose that $\alpha$, $\beta$, and $\gamma$ are the curve, dual curve and double dual on $S_{0, 4}$ shown in \cref{cuff_dual_double}.
Set $\Delta = \{\alpha, \beta, \gamma\}$.
The $\Delta$--coordinates of the identity is
\[
\Delta(\Id) = 
\begin{pmatrix}
0 & 2 & 2 \\
2 & 0 & 4 \\
2 & 4 & 0
\end{pmatrix}
\]
and, for $k > 0$, we have
\[
\Delta(T_\alpha^k) = 
\begin{pmatrix}
0 &  2     & 2 \\
2 & 4k     & 4k + 4 \\
2 & 4k - 4 & 4k
\end{pmatrix}
\inlineand
\Delta(T_\beta^k) = 
\begin{pmatrix}
4k     & 2 & 8k - 2 \\
2      & 0 & 4 \\
8k + 2 & 4 & 16k
\end{pmatrix}
\]
Applying the remark above, the transposes of these give $\Delta(T_\alpha^{-k})$ and $\Delta(T_\beta^{-k})$.
\end{example}

\begin{remark}
\label[remark]{construct_twist_delta_coordinate}
More generally, suppose that $\zeta$ is a short simple closed curve. 
Then there is an algorithm to compute the matrix $\Delta(T_\zeta^k)$; the running time is polynomial in $\log(k)$.
Furthermore, the entries of $\Delta(T_\zeta^k)$ are, eventually, linear functions of $k$.
Thus $\complexity{\Delta(T_\zeta^k)} = O(\log(k))$. 
See~\cite[Section~5]{SSS08} and~\cite[Proposition~22]{Thurston08} for discussions;
see~\cite[kernel/twist.py]{curver} for an implementation.
\end{remark}

\begin{corollary}
Suppose that $f_n, \ldots, f_2, f_1 \in F \cup F^{-1}$ is a list of generators and inverses of generators.
Suppose that $f = f_n \circ \cdots f_2 \circ f_1$ acts trivially on $H_1(S)$. 
Then $f$ is the identity if and only if $\Delta(f) = \Delta(\Id)$. \qed
\end{corollary}

So, to solve the word problem in the mapping class group, it suffices to compute the $\Delta$ coordinates of a class $f$ (given as a word over $F \cup F^{-1}$).
We do this via another divide-and-conquer.

\begin{algorithm}[ht!]
\caption{$\DCoord$}
\label[algorithm]{delta_coordinate}
\begin{algorithmic}[1]
\Require{A list $f_n, \ldots, f_2, f_1 \in F \cup F^{-1}$}
\Ensure{The matrix $\Delta(f_n \circ \cdots \circ f_2 \circ f_1)$}
    \IfThen{$n = 0$}{\Return $\Delta(\Id)$} \label{Line:base_0}
    \IfThen{$n = 1$}{\Return $\Delta(f_1)$} \label{Line:base_1}
    \State $k \gets n \fdiv 2$
    \State $M \gets \DCoord(f_n, \ldots, f_{k+1})$ \Comment{$g = f_n \circ \cdots \circ f_{k+1}$.} 
    \State $N \gets \DCoord(f_k, \ldots, f_1)$  \Comment{$h = f_k \circ \cdots \circ f_1$.} 
    \ForAll {$\delta \in \Delta$}
    \State $\alpha_\delta \gets (M_{\delta, \epsilon})_{\epsilon \in \Delta}$ \Comment{$\alpha_\delta = \Delta(g^{-1}(\delta))$.}
    \EndFor
    \ForAll {$\epsilon \in \Delta$}
    \State $\beta_{\epsilon} \gets (N_{\delta, \epsilon})_{\delta \in \Delta}$ \Comment{$\beta_\epsilon = \Delta(h(\epsilon))$.}
    \EndFor
    \State \Return $( \Call{FastIntersection}{\alpha_\delta, \beta_\epsilon} )_{\delta, \epsilon \in \Delta}$
    \Comment{Via \cref{ConvertDeltaTrain}.}
\end{algorithmic}
\end{algorithm}

\begin{theorem}
\label[theorem]{delta_coordinates_correctness}
Suppose that $f_n, \ldots, f_2, f_1 \in F \cup F^{-1}$ is a list of generators and inverses of generators. 
Suppose that $f = f_n \circ \cdots \circ f_2 \circ f_1$. 
Then \Cref{delta_coordinate} returns $\Delta(f)$.
Furthermore, \cref{delta_coordinate} runs in $O(n \log^3(n))$ time.
\end{theorem}

\begin{proof}
By induction on $n$, we have that $M = \Delta(g)$ and $N = \Delta(h)$.
Applying the definitions and \Cref{symmetry_of_delta} we have that the $\Delta$--coordinates of $\alpha_\delta$ and $\beta_\epsilon$ appear as rows and columns in these matrices: namely, $\alpha_\delta = \Delta(g^{-1}(\delta))$ and $\beta_\epsilon = \Delta(h(\epsilon))$.
From the definitions and the symmetry of geometric intersection we have that
\[
\Delta(f)_{\delta, \epsilon} = 
\intersection(\delta, f(\epsilon)) = 
\intersection(\delta, gh(\epsilon)) = 
\intersection(g^{-1}(\delta), h(\epsilon)) = 
\intersection(\alpha_\delta, \beta_\epsilon).
\]
Therefore \cref{delta_coordinate} returns $\Delta(f)$.

Finally, suppose that $t_1(n)$ bounds, from above, the running time needed for $\Call{DeltaCoordordinate}{f_n, \ldots, f_2, f_1}$.
Then for each $\delta, \epsilon \in \Delta$ we have that $\complexity{\Delta(\alpha_\delta)}, \complexity{\Delta(\beta_\epsilon)} = O(n)$.
In \cref{fast_intersection_correctness} we show that \cref{fast_intersection} (\textsc{FastIntersection}) computes $\intersection(\alpha_\delta, \beta_\epsilon)$ in $O(n \log^2(n))$ time.
Therefore
\[ t_1(n) \leq 2 t_1(n / 2) + O(n \log^2(n)) \]
and so $t_1(n) = O(n \log^3(n))$ by the master theorem.
\end{proof}

This proves \cref{WordProblem}, subject to the correctness of \cref{fast_intersection}.

\section{Train tracks}
\label[section]{sec:train_tracks}

We review some standard definitions and then discuss \emph{tight pairs} of train tracks, generalising work of Dynnikov~\cite{Dynnikov22}.

\subsection{Index}

Suppose that $Q$ is a \emph{region}: 
a compact riemannian surface with piecewise smooth boundary (perhaps empty).
We assume that at each non-smooth point of $\bdy Q$ the two adjacent arcs of the boundary make an interior angle of either zero or 90 degrees.
We call the former points \emph{(outward) cusps} and the latter \emph{(outward) corners}. 
We define the \emph{index} of $Q$ following~\cite[page~57]{Mosher03}:
\[
\ind(Q)
  \defeq \chi(Q) 
    - \frac{1}{2}(\textrm{number of cusps in $\bdy Q$})
    - \frac{1}{4}(\textrm{number of corners in $\bdy Q$})
\]
The second and third terms give a combinatorial version of the contribution of geodesic curvature in the Gauss--Bonnet formula.
Note that if $Q = R \sqcup R'$, then $\ind(Q) = \ind(R) + \ind(R')$.

With $Q$ a region as above, suppose now that $\alpha$ is a properly embedded simple arc, or closed curve, in $Q$.  
If $\alpha$ is an arc, then we suppose that its endpoints
\begin{itemize}
\item avoid the cusps and corners of $\bdy R$ 
and
\item are tangent or perpendicular to $\bdy R$.  
\end{itemize}
In a small abuse of notation, we define $\closure{Q - \alpha}$ to be the completion (in the induced path metric) of $Q - \alpha$.
We can now state the ``additivity'' property of the index~\cite[pages~57--58]{Mosher03}.

\begin{lemma}
\label[lemma]{Additivity}
With $Q$ and $\alpha$ as above, $\ind(\,\closure{Q - \alpha}\,) = \ind(Q)$. \qed
\end{lemma}

\subsection{Train tracks}

A train track $\tau$ in a surface $S$ is essentially an embedded graph with smoothings at its vertices; 
see~\cite[Section~8.9]{Thurston80} or~\cite{PennerHarer92, Mosher03} for further background.
Here is a definition suited to our needs.

\begin{definition}
\label[definition]{TrainTrack}
Suppose that $\tau \subset S$ is a closed subset. 
Suppose that $S(\tau)$ is a finite subset of $\tau$; 
these are the \emph{switches} of $\tau$.
The points of $\tau - S(\tau)$ are the \emph{branch points} of $\tau$.
We say that $\tau$ is a \emph{train track} if it satisfies the following.
\begin{enumerate}
\item 
Every branch point $x$ of $\tau$ has a disc neighbourhood $U_x$ (in $S$) so that $\tau \cap U_x$ is a smooth simple arc properly embedded in $U_x$.
\item 
\label{Itm:Switch}
Every switch $s \in S(\tau)$ has a disc neighbourhood $U_s$ (in $S$) so that $\tau_s = \tau \cap U_s$ is a smooth graph properly embedded in $U_s$.
The graph $\tau_s$ has one vertex of valence three (at $s$), three vertices of valence one (on $\bdy U_s)$, and three edges, called \emph{branch-ends}.
We require that the three components of $U_s - \tau_s$ have indices one-half, one-half, and zero, respectively. 
See \cref{Fig:switch}.
The component of $U_s - \tau_s$ with index zero is called the \emph{cusp} at $s$.
The branch-end not contained in the boundary of the cusp at $s$ is called \emph{large}; 
the two branch-ends contained in the boundary of the cusp are called \emph{small}.
\item 
Suppose that $C$ is a component of $S - \tau$. 
Suppose that $R$ is the region obtained by taking the completion of $C$ (in its induced path metric. 
Then $R$ has negative index.
\end{enumerate}
The components of $\tau - S(\tau)$ are called the \emph{branches} of $\tau$; 
we gather these into a set $B(\tau)$.
A branch is \emph{large} if both of its branch-ends are large;
a branch is \emph{small} if both of its branch-ends are small;
a branch is \emph{mixed} if it has one large and one small branch-end.
\end{definition}

\begin{figure}
\centering
\switch{black}{black}{black}
\caption{A switch}
\label{Fig:switch}
\end{figure}

Note that a train track may have connected components without switches.  
These components are necessarily simple closed curves, no two of which are isotopic to each other (or to a component of $\bdy S$).
Additionally, train tracks in $S$ have at most $\bee \defeq 6(3g - 3 + b)$ branches. 

\subsection{Weightings}

\begin{definition}
Suppose that $\tau$ is a train track.
A \emph{weighting} (on $\tau$) is a function $\mu \from B(\tau) \to \ZZ$. 
We denote the space of weightings on $\tau$ by $W(\tau)$.
We refer to the pair $(\tau, \mu)$ as a \emph{weighted train track}.
\end{definition}

A weighting $\mu$ on $\tau$ gives a multi-curve $C(\mu)$ exactly when:
\begin{itemize}
\item (non-negativity) $\mu(b) \geq 0$ for each $b \in B(\tau)$ and
\item (switch equality) for any switch $s \in S(\tau)$, if $a$, $b$, and $c$ are the branch-ends at $s$, with $a$ large and with $b$ and $c$ small, then we have $\mu(a) = \mu(b) + \mu(c)$.
\end{itemize}
We build $C(\mu)$ by taking $\mu(b)$ arcs parallel to $b$ and gluing ends according to the switch equality.  
We use $V(\tau) \subset W(\tau)$ to denote the cone of such weightings.

\subsection{Pairs of train tracks}
\label{Sub:train_track_pairs}




\begin{definition}
\label[definition]{tight}
Suppose that $\sigma$ and $\tau$ are train tracks.
The pair $(\sigma, \tau)$ is \emph{tight} if it satisfies the following.
\begin{itemize}
\item 
Every point $x$ of $\sigma \cap \tau$ is (exactly) one of the following: 
\begin{itemize}
\item 
a \emph{crossing}: there is a disc neighbourhood $U_x$ of $x$ in $S$ so that $\sigma \cap U_x$ and $\tau \cap U_x$ are simple arcs, properly embedded in $U_x$, meeting exactly once, transversely, at $x$. 
The four components of $U_x - (\sigma \cup \tau)$ are \emph{corners}.
See \cref{Fig:crossing}.
\item
a \emph{tangency}: there is a disc neighbourhood $U_x$ of $x$ in $S$ so that $\sigma \cap U_x = \tau \cap U_x$ is a single arc.
See \cref{Fig:tangency}.
\item 
a \emph{shared switch}: there is a disc neighbourhood $U_x$ of $x$ in $S$ so that $x$ is a switch of $U_x \cap (\sigma \cup \tau)$, as described in \cref{TrainTrack}(\ref{Itm:Switch}), $x \in S(\sigma)$ and $x \in S(\tau)$.
See \cref{Fig:shared_switch}.
\item
a \emph{$\sigma$--switch} (respectively \emph{$\tau$--switch}): there is a disc neighbourhood $U_x$ of $x$ in $S$ so that $x$ is a switch of $U_x \cap (\sigma \cup \tau)$, $x \in S(\sigma)$ and $x \not \in S(\tau)$ (resp. $x \not \in S(\sigma)$ and $x \in S(\tau)$).
See \cref{Fig:sigma-switch}.
\item 
a \emph{divergence}: there is a disc neighbourhood $U_x$ of $x$ in $S$ so that $x$ is a switch of $U_x \cap (\sigma \cup \tau)$, $x \not \in S(\sigma)$ and $x \not \in S(\tau)$.
See \cref{Fig:divergence}.
\end{itemize}
\item 
Suppose that $C$ is a component of $S - (\sigma \cup \tau)$. 
Suppose that $R$ is the completion of $C$ (in its induced path metric). 
Then the region $R$ must be \emph{legal}: that is, either
\begin{itemize}
\item $\ind(R) < 0$, or
\item $R$ matches one of the cases shown in \cref{Fig:high_index_regions}. \qedhere
\end{itemize}
\end{itemize}
\end{definition}

\begin{figure}
\centering
\begin{subfigure}[b]{0.3\textwidth}
    \centering
    \input{tikz/crossing}
    \caption{A crossing point}
    \label{Fig:crossing}
\end{subfigure}
\hfill
\begin{subfigure}[b]{0.3\textwidth}
    \centering
    \input{tikz/tangency}
    \caption{A tangency point}
    \label{Fig:tangency}
\end{subfigure}
\hfill
\begin{subfigure}[b]{0.3\textwidth}
    \centering
    \switch{black}{black}{black}
    \caption{A shared switch}
    \label{Fig:shared_switch}
\end{subfigure}

\begin{subfigure}[b]{0.3\textwidth}
    \centering
    \switch{black}{black}{red}
    \caption{A $\sigma$--switch point}
    \label{Fig:sigma-switch}
\end{subfigure}
\hfill
\begin{subfigure}[b]{0.3\textwidth}
    \centering
    \switch{black}{black}{blue}
    \caption{A $\tau$--switch point}
    \label{Fig:tau-switch}
\end{subfigure}
\hfill
\begin{subfigure}[b]{0.3\textwidth}
    \centering
    \switch{black}{red}{blue}
    \caption{A divergence point}
    \label{Fig:divergence}
\end{subfigure}

\caption{Neighbourhoods of points of $\sigma \cap \tau$ (up to reflection). Here $\sigma$ is shown in red, $\tau$ in blue and $\sigma \cap \tau$ in black.}
\end{figure}
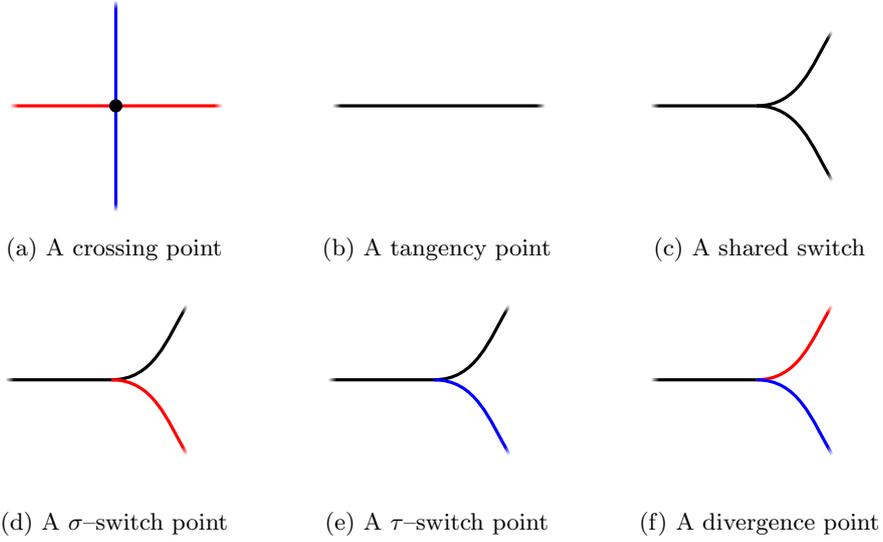

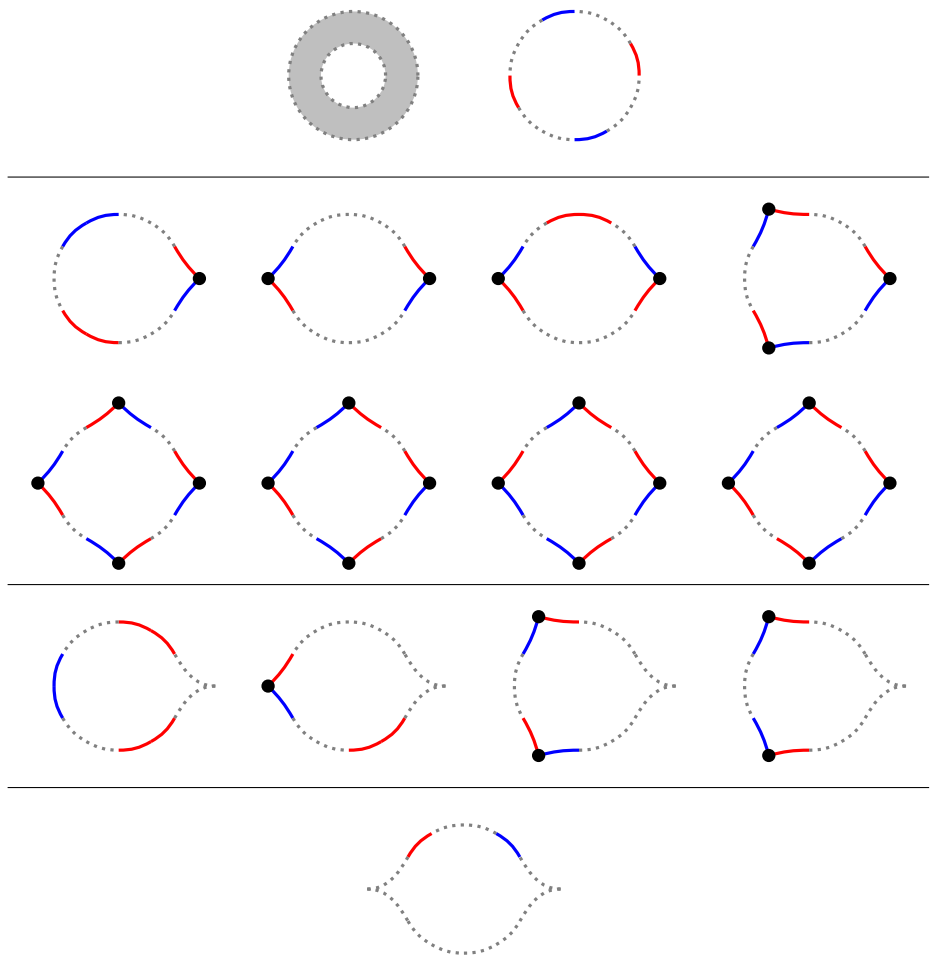
\begin{figure}

\centering
\begin{tikzpicture}[scale=0.85,very thick]
\draw [transparent] (0,0) circle (1.5);  
\draw [fill=lightgray, dotted, draw=gray] (0,0) circle (1);
\draw [fill=white, dotted, draw=gray] (0,0) circle (0.5);
\end{tikzpicture} 
\tikzdisk{180,180,180,180,180,180,180,180,180,180,180,180}{R,G,G,B,G,G,R,G,G,B,G,G}
\hrule

\tikzdisk{90,180,180,180,180,180,180,180,180,180,180,180}{R,G,G,B,B,G,G,R,R,G,G,B}
\tikzdisk{90,180,180,180,180,180,90,180,180,180,180,180}{R,G,G,G,G,B,R,G,G,G,G,B}
\tikzdisk{90,180,180,180,180,180,90,180,180,180,180,180}{B,G,R,R,G,B,R,G,G,G,G,R}
\tikzdisk{90,180,180,180,90,180,180,180,90,180,180,180}{R,G,G,R,B,G,G,R,B,G,G,B}

\tikzdisk{90,180,180,90,180,180,90,180,180,90,180,180}{R,G,B,R,G,B,R,G,B,R,G,B}
\tikzdisk{90,180,180,90,180,180,90,180,180,90,180,180}{R,G,R,B,G,B,R,G,B,R,G,B}
\tikzdisk{90,180,180,90,180,180,90,180,180,90,180,180}{R,G,R,B,G,R,B,G,B,R,G,B}
\tikzdisk{90,180,180,90,180,180,90,180,180,90,180,180}{R,G,R,B,G,B,R,G,R,B,G,B}
\hrule

\tikzdisk{0,180,180,180,180,180,180,180,180,180,180,180}{G,R,R,G,G,B,B,G,G,R,R,G}
\tikzdisk{0,180,180,180,180,180,90,180,180,180,180,180}{G,G,G,G,G,R,B,G,G,R,R,G}
\tikzdisk{0,180,180,180,90,180,180,180,90,180,180,180}{G,G,G,R,B,G,G,R,B,G,G,G}
\tikzdisk{0,180,180,180,90,180,180,180,90,180,180,180}{G,G,G,R,B,G,G,B,R,G,G,G}
\hrule

\tikzdisk{0,180,180,180,180,180,0,180,180,180,180,180}{G,B,G,G,R,G,G,G,G,G,G,G}

\caption{Legal high index regions (up to reflection and interchanging $\sigma$ and $\tau$). 
Here $\sigma$ is shown in red, $\tau$ in blue, $\sigma \cap \tau$ in black and unknown sections of $\sigma \cup \tau$ in dotted gray.}
\label{Fig:high_index_regions}
\end{figure}

\begin{remark}
\label[remark]{colour_changes}
A region $R$ is legal if and only if the number of colour changes in $\bdy R$
\begin{itemize}
\item is at least one if $R$ is a cusped bigon (has index zero and no corners) and
\item is at least $4 \cdot \ind(R)$ if $R$ is not a cusped bigon. \qedhere
\end{itemize}
\end{remark}

We denote the set of crossing points by $\sigma \pitchfork \tau$.
A maximal subarc of $\sigma \cap \tau$ consisting only of tangency points is called a \emph{shared branch} of $(\sigma, \tau)$.
We denote the set of shared branches of $(\sigma, \tau)$ by $B(\sigma \cap \tau)$.
Finally, we call a shared branch \emph{isolated} if both of its endpoints are divergence points.
See cases~1 and~2 in \cref{splitting}.

\begin{proposition}
\label[proposition]{tight_subtracks}
Suppose that $(\sigma, \tau)$ is a tight pair of train tracks.
If $\sigma'$ is a subtrack of $\sigma$ and $\tau'$ is a subtrack of $\tau$ then $(\sigma', \tau')$ is also a tight pair of train tracks.
\end{proposition}

\begin{proof}
Without loss of generality we may assume that $\sigma$ and $\sigma'$ differ by a single branch $b$ and that $\tau = \tau'$.
Suppose that $Q$ is a region of $S - (\sigma' \cup \tau')$.
If $\ind(Q) < 0$ then we are done.

So suppose that $\ind(Q) \geq 0$.
If $Q$ is an annulus then it matches one of the cases shown in \cref{Fig:high_index_regions} and so we are done.

So suppose that $Q$ is (topologically) a disk.
If $b$ is disjoint from the interior of $Q$ then $\bdy Q$ must contain at least as many colour changes as $\bdy (Q \cup b)$. 
Therefore, by \cref{colour_changes}, $Q$ must also match one of the cases shown in \cref{Fig:high_index_regions} and so we are done.

So suppose that $b$ crosses through the interior of $Q$.
It is sufficient to prove the result when $b$ passes through the interior of $Q$ exactly once, splitting $Q$ into two subregions $R$ and $R'$.
Hence $Q$ and $b$ must match one of the four cases shown in \cref{Fig:branch_cross_disk}.

\begin{figure}
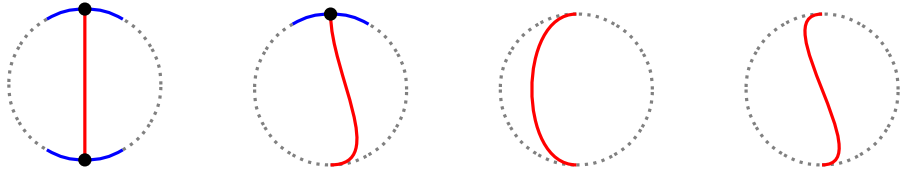

\centering
\begin{subfigure}[b]{0.2\textwidth}
    \centering
    \tikzdisksplit{180,180,180,180,180,180,180,180,180,180,180,180}{G,G,B,B,G,G,G,G,B,B,G,G}{4}{10}{0}{0} 
    \label{Fig:crossing-I}
\end{subfigure}
\hfill
\begin{subfigure}[b]{0.2\textwidth}
    \centering
    \tikzdisksplit{180,180,180,180,180,180,180,180,180,180,180,180}{G,G,B,B,G,G,G,G,G,G,G,G}{4}{10}{0}{-1}
    \label{Fig:crossing-J}
\end{subfigure}
\hfill
\begin{subfigure}[b]{0.2\textwidth}
    \centering
    \tikzdisksplit{180,180,180,180,180,180,180,180,180,180,180,180}{G,G,G,G,G,G,G,G,G,G,G,G}{4}{10}{-1}{1}
    \label{Fig:crossing-C}
\end{subfigure}
\hfill
\begin{subfigure}[b]{0.2\textwidth}
    \centering
    \tikzdisksplit{180,180,180,180,180,180,180,180,180,180,180,180}{G,G,G,G,G,G,G,G,G,G,G,G}{4}{10}{-1}{-1}
    \label{Fig:crossing-S}
\end{subfigure}

\caption{The branch $b$ can pass through $Q$ forming an \texttt{I}, \texttt{J}, \texttt{C} or \texttt{S}.}
\label{Fig:branch_cross_disk}
\end{figure}

We exhaustively enumerate all possibilities for $R$ and $R'$ and check, in each case, that the resulting $Q$ satisfies \cref{colour_changes}.
See \cref{cut_disk_verification} for our implementation.
\end{proof}

\begin{remark}
The exhaustive search at the end of the proof of \cref{tight_subtracks} is not particularly large -- we have carried out a version spanning six hand-written pages.  
However, it is a somewhat long proof by cases.
This would be eliminated if we had a ``chromatic index'' that accounts for colour changes (each an additional minus one-quarter) which is ``local'' -- which correctly distributes index to certain regions of $\sigma \cap \tau$ -- and which is additive. 
However, we have failed to find such an index.
\end{remark}

\begin{definition}[{\cite[page~19]{PennerHarer92}}]
Suppose that $(\sigma, \tau)$ is a tight pair of train tracks. 
If all points of $\sigma \cap \tau$ are crossing points then we say that $(\sigma, \tau)$ is a \emph{crossing} pair.
\end{definition}

Suppose that $(\sigma, \tau)$ is a tight pair of train tracks.
Following~\cite[page~197]{PennerHarer92}, define a matrix $M \from B(\sigma) \times B(\tau) \to \ZZ$ by $M(a, b) \defeq |a \pitchfork b|$.
That is, $M$ records the number of crossing points between each pair of branches. 
This gives a bilinear map $\pair{\cdot, \cdot} \from W(\sigma) \cross W(\tau) \to \ZZ$ via:
\[
\pair{\mu, \nu} \defeq \sum_{\substack{a \in B(\sigma)\\ b \in B(\tau)}} \mu(a) \, M(a, b) \, \nu(b)
\]

\begin{lemma}[{\cite[Remark, page~197]{PennerHarer92}}]
\label[lemma]{pair_is_intersection}
Suppose that $(\sigma, \tau)$ is a crossing pair of train tracks.
If $\mu \in V(\sigma)$ and $\nu \in V(\tau)$ then $\pair{\mu, \nu} = \iota(C(\mu), C(\nu))$. \qed
\end{lemma}

\begin{definition}
Suppose that $(\sigma, \tau)$ is a tight pair of train tracks.
If they have no isolated shared branches then we say that $(\sigma, \tau)$ is a \emph{clean} pair.
\end{definition}

\begin{remark}
A clean pair of train tracks has at most $2 \bee$ shared branches.
\end{remark}

We will often write $\Tau = (\sigma, \mu, \tau, \nu)$ for a pair of weighted train tracks $(\sigma, \mu)$ and $(\tau, \nu)$.
We say that such a pair is tight / crossing / clean if the underlying pair of train tracks $(\sigma, \tau)$ are.
Finally, if $e$ is a shared branch of $\sigma \cap \tau$ which is contained within $a \in B(\sigma)$ and $b \in B(\tau)$ then we define $\mu(e) \defeq \mu(a)$ and $\nu(e) \defeq \nu(b)$.

\subsection{Standard pairs of train tracks}

Recall that $\Delta = P \cup Q \cup R$ is our chosen set of cuffs, duals, and double duals on $S$.

\begin{definition}
Suppose that $\sigma$ and $\tau$ are a pair of train tracks in $S$.
Suppose that inside every annulus component of $N(P)$ we have that $\sigma$ and $\tau$ are as shown in \cref{Fig:standard_annulus}.
Suppose that inside every pair of pants component of $S - N(P)$ these are as shown in \cref{Fig:standard_pants}.
Then we say that the pair $(\sigma, \tau)$ is \emph{standard} (with respect to $\Delta$).
\end{definition}

\begin{figure}[htb]
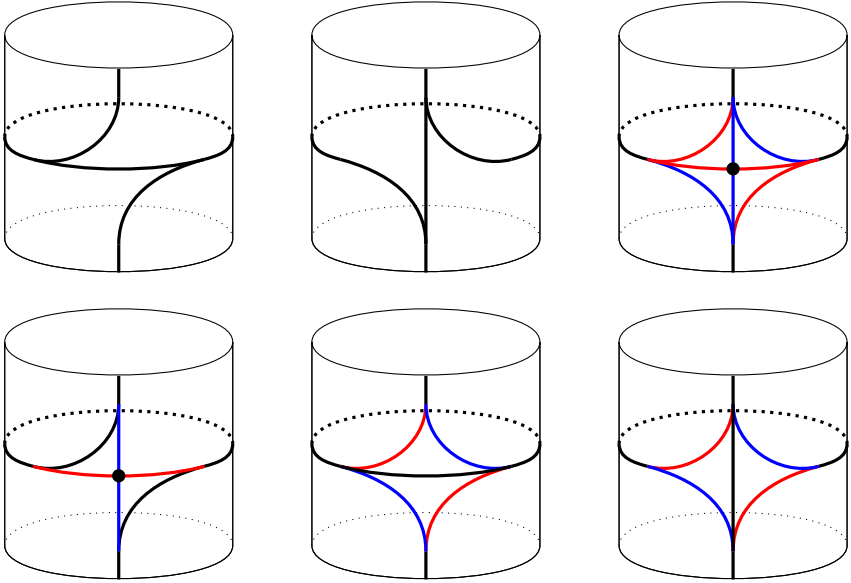

\centering
\hfill
\begin{minipage}{0.3\textwidth}
\standardannulus{black}{transparent}{black}{transparent}
\end{minipage}
\hfill
\begin{minipage}{0.3\textwidth}
\standardannulus{transparent}{black}{transparent}{black}
\end{minipage}
\hfill
\begin{minipage}{0.3\textwidth}
\standardannulus{red}{blue}{red}{blue}
\end{minipage}
\hfill

\hfill
\begin{minipage}{0.3\textwidth}
\standardannulus{black}{transparent}{red}{blue}
\end{minipage}
\hfill
\begin{minipage}{0.3\textwidth}
\standardannulus{red}{blue}{black}{transparent}
\end{minipage}
\hfill
\begin{minipage}{0.3\textwidth}
\standardannulus{red}{blue}{transparent}{black}
\end{minipage}
\hfill

\caption{The standard pairs of train tracks in an annulus (up to rotation, reflection and interchanging $\sigma$ and $\tau$).}
\label{Fig:standard_annulus}
\end{figure}

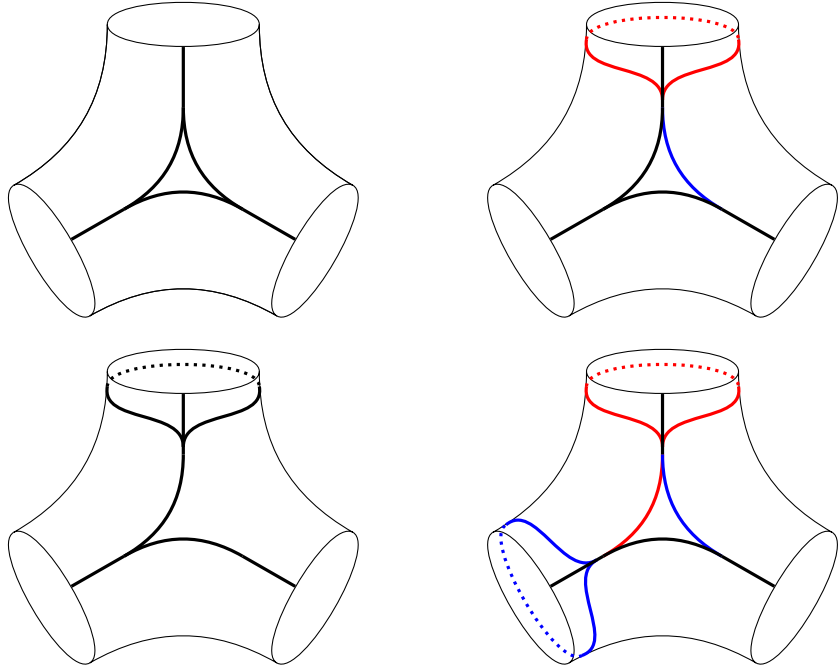
\begin{figure}[htb]
\begin{minipage}{0.48\linewidth}
\centering
\input{tikz/standard_tri_tri}
\end{minipage}
\hfill
\begin{minipage}{0.48\linewidth}  
\centering
\input{tikz/standard_tri_loop}
\end{minipage}
\begin{minipage}{0.48\linewidth}
\centering
\input{tikz/standard_loop_loop}
\end{minipage}
\hfill
\begin{minipage}{0.48\linewidth}
\centering
\input{tikz/standard_loop_loop_offset}
\end{minipage}
\caption{The standard pairs of train tracks in a pair of pants (up to rotation, reflection, and interchanging $\sigma$ and $\tau$).}
\label{Fig:standard_pants}
\end{figure}

There are $14^{3g-3} \cdot 16^{2g-2}$ standard pairs of tracks on $S_{g, 0}$.

\begin{lemma}
\label[lemma]{standard_tight}
Standard pairs of train tracks are tight.
\end{lemma}

\begin{proof}
We exhaustively enumerated all possible combinations of pants and annuli and checked that in each case the resulting complementary regions satisfied \cref{colour_changes}.
See \cref{standard_pairs_verification} for an implementation.
\end{proof}





\begin{proposition}
\label[proposition]{ConvertDeltaTrain}
Suppose that $\alpha, \beta \in C(S)$ are multicurves.
Given $\Delta(\alpha)$ and $\Delta(\beta)$, we can compute a tight pair of weighted train tracks $(\sigma, \mu, \tau, \nu)$ such that $C(\mu) = \alpha$ and $C(\nu) = \beta$.
Furthermore, this can be done in $O(n)$ time where $n = \complexity{\Delta(\alpha)} + \complexity{\Delta(\beta)}$.
\end{proposition}

\begin{proof}
First, given $\Delta(\alpha)$ and $\Delta(\beta)$, their intersections with the cuffs $P$ determines (possibly more than one) standard pair of train track in each pair of pants and its weights (some of which may be zero).

Second, for each curve in $P$ we can compute the $(m, s, t)$--coordinates of $\alpha$ and $\beta$ (with respect to the dual and double dual curves in $Q$ and $R$) \cite[Proposition~C.1]{FLP12} \cite[Page~120]{ThurstonTravaux79}.
These $(m, s, t)$--coordinates determine (possibly more than one) standard pair of train tracks in each annulus and its weights (some of which may be zero).

Removing the branches of weight zero, the resulting pair of train tracks is tight by \cref{tight_subtracks} and \cref{standard_tight}.

By \cref{ConvertDeltaMST} the resulting pair of train tracks and its weights can be computed in $O(n)$ time.
\end{proof}

\section{Improving}

In this section we describe moves that can be applied to a tight pair of weighted train tracks $\Tau = (\sigma, \mu, \tau, \nu)$.
These moves are analogous to those used by Agol--Hass--Thurston \cite[Section~4]{AHT06} for simplifying weighted train tracks. 
See also the work of Erickson--Nayyeri~\cite{EricksonNayyeri13}.

We will show how the correct sequence of moves can be used to reach a crossing pair of weighted train tracks.
In particular, this process can be applied starting from a substandard pair of weighted train tracks since these are tight by \cref{tight_subtracks}.

When $(\sigma, \mu)$ and $(\tau, \nu)$ correspond to multi-curves, these moves preserve the geometric intersection number of $C(\mu)$ and $C(\nu)$.
Thus we can then immediately calculate $\intersection(C(\mu), C(\nu))$ via $\pair{\cdot, \cdot}$ once we have reached a crossing pair.

From now on we will assume that each train track $\tau$ comes with a chosen ordering of its branches $(b_1, b_2, \ldots, b_k)$.
This allows us to explicitly write weightings on $\tau$ as vectors and various linear transformations between them as matrices.
We use these to measure the complexity of weightings and such linear transformations.
Since train tracks have at most $\bee$ branches, all of the following matrices are at most $\bee \times \bee$.

\subsection{Splitting}

We now describe our first combinatorial move that can be applied to a tight pair of weighted train tracks $\Tau = (\sigma, \mu, \tau, \nu)$: splitting.
    
Let $e$ be a shared branch of $\Tau$ as shown in the Source column of \cref{splitting} or \cref{conditional_splitting}.
There we draw $\sigma$ in red, $\tau$ in blue and the shared branches in black.
We refer to such a shared branch as \emph{splittable}.
Then the result of \emph{splitting} $\Tau$ along $e$ is a new pair of weighted train tracks $\Tau' = (\sigma', \mu', \tau', \nu')$ shown in the corresponding row of the Target column.

There are linear transformations $A$ and $B$, depending only on the case that we are in, such that 
\[ \mu' = A(\mu) \inlineand \nu' = B(\nu). \]
We record these via an \emph{update rule} $U = (\sigma', \tau', A, B)$ of \emph{complexity} $\complexity{U} \defeq \complexity{A} + \complexity{B}$.
We denote the application of an update rule via
\[ U(\Tau) \defeq (\sigma', A(\mu), \tau', B(\nu)). \]
and the composition with another update rule $U' = (\sigma'', \tau'', A', B')$ via
\[ U' \circ U \defeq (\sigma'', \tau'', A' \circ A, B' \circ B). \]

\subsection{Untwisting}
\label{Sec:untwisting}

We now describe our second combinatorial move that can be applied to a tight pair of weighted train tracks $\Tau = (\sigma, \mu, \tau, \nu)$: untwisting.

An oriented train cycle $c$ in $\sigma \cap \tau$ is \emph{compatibly combed} in both $\sigma$ and $\tau$ it exits out of the left (or equivalently right) small branch-ends at every switch that it enters.
For example, see \cref{Fig:compatible_combed}.
We note that $\sigma \cap \tau$ has at most $2\bee$ compatibly combed train cycles.

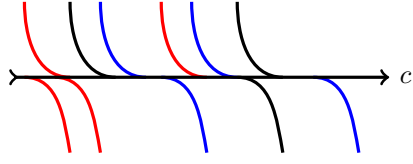
\begin{figure}[ht]
\centering
\input{tikz/combed_compatible}
\caption{A (left) compatibly combed train cycle $c$ in $\sigma \cap \tau$.}
\label{Fig:compatible_combed}
\end{figure}

Suppose that $c$ is such a compatibly combed train cycle in $\sigma \cap \tau$.
Let $X$ denote the set of branches of $\sigma$ that are contained in $c$ and let $Y$ denote the set of branches of $\sigma$ which meet $c$ but are not contained in it.
Let $X'$ and $Y'$ be defined analogously for $\tau$.
Then 
\[ p \defeq 2 \min_{a \in X} \mu(a) \fdiv \sum_{a \in Y} \mu(a) \]
and
\[ q \defeq 2 \min_{b \in X'} \nu(b) \fdiv \sum_{b \in Y'} \nu(b) \]
record the number of times that $(\sigma, u)$ and $(\tau, v)$ wind around $c$ respectively.
Therefore we may untwist $r \defeq \min(p, q)$ times without changing the combinatorics.
We refer to $r$ as the \emph{rotation number} of $\Tau$ along $c$.

The result of \emph{untwisting} $\Tau$ along $c$ is $\Tau' \defeq (\sigma, A(\mu), \tau, B(\nu))$ where
\[ A(\mu)(a) \defeq \begin{cases}
\mu(a) - r \sum_{y \in Y}  \mu(y) & \textrm{if $a \in X$} \\
\mu(a) & \textrm{otherwise}
\end{cases}
\]
and
\[ B(\nu)(b) \defeq \begin{cases}
\nu(b) - r \sum_{y \in Y'} \nu(y) & \textrm{if $b \in X'$} \\
\nu(b) & \textrm{otherwise}
\end{cases}
\]
are linear transformations.
Again, we write this as an update rule $U = (\sigma, \tau, A, B)$.

\begin{remark}
Although $c$ is also a compatibly combed train cycle in $\Tau'$, its rotation number in $\Tau'$ is $0$ by construction.
\end{remark}

\subsection{Separating}

We now describe our third combinatorial move that can be applied to a tight pair of weighted train tracks $\Tau = (\sigma, \mu, \tau, \nu)$: separating.

A train cycle $c$ in $\sigma \cap \tau$ is \emph{incompatibly combed} if in $\sigma$ it uses the left (respectively right) small branch-ends at every switch that it enters and in $\tau$ it exits uses the left (resp. right) small branch-ends at every switch that it enters.
For example, see \cref{Fig:incompatible_combed}.
We again note that $\sigma \cap \tau$ has at most $2\bee$ incompatibly combed train cycles.

\begin{figure}[ht]
\centering
\begin{subfigure}{0.4\textwidth}
    \input{tikz/combed_incompatible}
    \caption{Before.}
    \label{Fig:incompatible_combed}
\end{subfigure}
\quad
\begin{subfigure}{0.4\textwidth}
    \input{tikz/combed_separate}
    \caption{After.}
    \label{Fig:combed_separated}
\end{subfigure}
\caption{An incompatibly combed train cycle $c$ in $\sigma \cap \tau$.}
\label{fig:incompatibly_combed}
\end{figure}
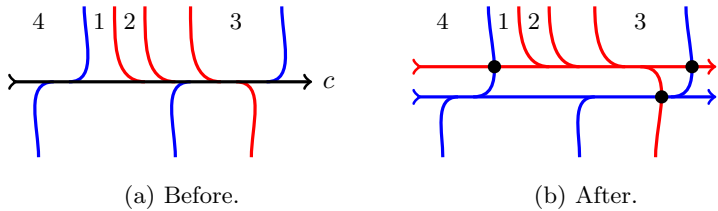

Suppose $c$ is such an incompatibly combed train cycle in $\sigma \cap \tau$.
The result of \emph{separating} $\Tau$ along $c$ is
\[ \Tau' \defeq (\sigma', \mu, \tau', \nu) \]
where $\sigma'$ and $\tau'$ are obtained by pushing $\sigma$ and $\tau$ off of different sides of $c$.
For example, see \cref{Fig:combed_separated}.

Again, we write this as an update rule $U = (\sigma', \tau', \Id, \Id)$.

\begin{remark}
If a train cycle $c$ in $\sigma \cap \tau$ does not contain any switches then it is incompatibly combed.
Thus $c$ can be separated into two disjoint cycles, one in $\sigma$ and one in $\tau$.
\end{remark}

\begin{proposition}
\label[proposition]{remains_tight}
Suppose that $\Tau$ is a tight pair of weighted train tracks.
Suppose that $\Tau'$ is the result of splitting, untwisting, or separating $\Tau$.
Then $\Tau'$ is again tight.
\end{proposition}

\begin{proof}
We consider each of the three possible types of moves: splitting, untwisting, and separating. 

Suppose that $\Tau'$ is the result of splitting $\Tau$ at the edge $e$.
All splits (up to various symmetries) are shown in \cref{splitting} and \cref{conditional_splitting}.
Let $R$ and $Q$ be the regions of $\Tau$ containing the cusps at the ends of $e$. 
Let $R'$ and $Q'$ be the regions of $\Tau'$ induced by $R$ and $Q$. 
There are three subcases: the split does not change the topology of any region, the split is central, or the split introduces a crossing.

\begin{itemize}
\item
Suppose that the split on $e$ does not change the topology of any region. 
Then the two cusps at $e$ give cusps of $\Tau'$. 
We deduce that $R$ is homeomorphic to $R'$.
Thus $\ind(R') = \ind(R)$. 
Also, $R'$ has cusps in the same locations as $R$; 
thus they have the same number of colour changes. 
Thus the legality of $R$ implies that of $R'$. 
The same holds for $Q$ and $Q'$. 

\item
Suppose that the split on $e$ is central.  
Suppose first that $R = Q$.  
Then $R' = Q'$ has topology (at least that of an annulus) and so is legal. 
Suppose instead that $R$ and $Q$ are distinct. 
So the split merges $R$ and $Q$ to form the region $R' = Q'$.
By \cref{Additivity}, and since two corners have the same total index as one cusp, we find that $\ind(R') = \ind(R) + \ind(Q)$. 
Also, $R'$ has at least as many colour changes as $R$ and $Q$ combined. 
Thus $R'$ is legal. 

\item
Suppose that the split on $e$ produces a crossing in $\Tau'$.
If $R = Q$ then we have that $R' = Q'$. 
The split at $e$ causes two cusps (in $R$) to become corners (in $R'$).
Thus $\ind(R') = \ind(R) + 1/2$. 
Also, $R'$ has the same number of colour changes as $R$. 
Recall that $R$ was legal.  
There are two subsubcases: 
$R$ either appears in \cref{Fig:high_index_regions} or $R$ has negative index. 
Suppose that $R$ appears in the figure; 
thus $R$ is a cusped bigon with at least one colour change. 
In this case $R'$ has two corners and at least one colour change, so is legal. 
Suppose that $R$ has negative index. 
Thus $R'$ has index either zero or one-quarter. 
Since $R'$ has corners, it is a disk. 
If $\ind(R') = 0$ then $R'$ either has two corners and a cusp or has four corners.  
If $\ind(R') = 1/4$ then $R'$ has three corners. 
In any case, $R'$ appears in \cref{Fig:high_index_regions}. 
The subcase where $R \neq Q$ is similar, except the change of index is only one-quarter (instead of one-half). 
\end{itemize}

Suppose that $\Tau'$ is the result of untwisting $\Tau$.
Then the underlying train tracks do not change.
So all regions of $\Tau'$ are legal. 

Suppose, finally, that $\Tau'$ is the result of separating $\Tau$ along a cycle $c$. 
Consulting \cref{fig:incompatibly_combed} we find that in a neighbourhood of $c$ we have an annulus $A$ cobounded by parallel copies of $c$ in $\sigma'$ and $\tau'$.
Suppose that $R'$ is a region of $\Tau'$.
If $R'$ lies inside of $A$ then $R'$ is either the whole of $A$ or has one cusp and two corners.
In either case $R'$ is legal (\cref{tight}).

Suppose instead that $R'$ lies outside of $A$.
The region $R'$ can meet the boundary of $A$ in four different ways; 
these are enumerated in \cref{fig:incompatibly_combed}.
Let $R$ be the region of $\Tau$ inducing $R'$.
Suppose that $R'$ meets the boundary of $A$ in $n$ arcs, say $(\alpha_k)_{k = 0}^{n-1}$.
We build a sequence 
\[
R = R_0, R_1, \ldots, R_k, \ldots, R_n = R'
\]
of regions by modifying the boundary of $R_k$ in a small neighbourhood of $\alpha_k$, in one of the four ways allowed by \cref{fig:incompatibly_combed}.
We induct on $k$ to prove that $R_n = R'$ is legal. 
So, suppose that $R_k$ is legal. 
\begin{itemize}
\item 
Suppose that $R_k$ and $R_{k + 1}$ differ by a piece of type two or of type four.
Then they have the same index and the same number of colour changes.  
\item
Suppose that $R_k$ and $R_{k + 1}$ differ by a piece of type three. 
Then $R_{k + 1}$ has one less colour change than $R_k$, but has one more corner.
\item 
Suppose that $R_k$ and $R_{k + 1}$ differ by a piece of type one.
Then the index increases by one-quarter; 
however the arcs of $\bdy R_k$ adjacent to $\alpha_k$ have different colours.  
Thus $R_{k + 1}$ has either additional cusps or corners or has at least one colour change (or both). 
\end{itemize}
Applying \cref{colour_changes}, we find that $R_{k + 1}$ is legal, as desired.
\end{proof}

\begin{remark}
\cref{remains_tight} should be contrasted with Lemma~10.4 of \cite{Dynnikov22}. 
Our version of tightness allows us to always be considering tracks embedded in $S$. 
This allows us to control the topology of their intersection and to deal with the case where $S$ is closed. 
\end{remark}

\subsection{Complexities}

We define several complexities that will be used to prove correctness of the algorithms which follow and bound their running times.

\begin{definition}
Suppose that $\Tau = (\sigma, \mu, \tau, \nu)$ is a clean pair of weighted train tracks.
We define its \emph{tightness} $\tightness{\Tau}$ to be the number of shared branches plus the number of shared switches of $(\sigma, \tau)$.
\end{definition}

\begin{lemma}
For any clean pair of weighted train tracks $\Tau$ we have that $0 \leq \tightness{\Tau} \leq 3 \bee$. 
Furthermore, $\tightness{\Tau} = 0$ if and only if $\Tau$ is a crossing pair. \qed
\end{lemma}

We note that $\tightness{\Tau}$ does not increase when apply a move and in fact only remains constant if we apply an untwist or a split of type 6, 7, 8a, 10a, or 11a (or the horizontal mirror of 11a).

\begin{definition}
Suppose that $\Tau = (\sigma, \mu, \tau, \nu)$ is a clean pair of weighted train tracks.
The \emph{complexity} of $\Tau$ is:
\[
\complexity{\Tau} \defeq \sum_{a \in B(\sigma)} \complexity{\mu(a)} + \sum_{b \in B(\tau)} \complexity{\nu(b)}
\]
The \emph{$L^1$--shared size} and \emph{$L^\infty$--shared size} of $\Tau$ are:
\[ 
\sharedsize{\Tau}_1 \defeq \sum_{\mathclap{e \in B(\sigma \cap \tau)}} \; \complexity{\mu(e)} + \complexity{\nu(e)}
\inlineand
\sharedsize{\Tau}_\infty \defeq \max_{\mathclap{e \in B(\sigma \cap \tau)}} \complexity{\mu(e)} + \complexity{\nu(e)}
\]
respectively.
Using these we define the \emph{shared size} of $\Tau$ to be:
\[ 
\sharedsize{\Tau} \defeq \sharedsize{\Tau}_1 + 2 \cdot \tightness{\Tau} \cdot \sharedsize{\Tau}_\infty. \qedhere
\]
\end{definition}



\subsection{Shortening}

Together these three moves allow us to reduce a clean pair of weighted train tracks to a crossing pair.

\begin{definition}
Suppose that $\Tau = (\sigma, \mu, \tau, \nu)$ is a clean pair of weighted train tracks.
We define $\Lambda(\Tau)$ to be the subtrack of $\sigma \cap \tau$ consisting of the shared branches on which $\complexity{\mu(e)} + \complexity{\nu(e)}$ is maximal.
\end{definition}

We note that at every switch at most two of the incident branches are in $\Lambda(\Tau)$, hence $\Lambda(\Tau)$ is a disjoint union of intervals and loops.

\begin{definition}
Suppose that $\Tau$ is a clean pair of weighted train tracks.
Suppose that $s$ is a switch where the large branch-end is in $\Lambda(\Tau)$.
If the left (respectively right, neither) small branch-end of $s$ is in $\Lambda(s)$ then let $\length{s}$ denote the length of the train path which follows $\Lambda(\Tau)$ out of the large branch-end of $s$ until it exits out of the right (respectively left, either) small branch-end of a switch.
For example, see \cref{Fig:switch_length}.
Note that $\ell(s)$ may be $\infty$ if this train path never exits out of the right small branch-end of a switch.

For a branch $b$ which is incident to switches $s$ and $s'$ we define $\length{b}$ to be the minimum of $\length{s}$ and $\length{s'}$.

We define $\length{\Tau}$ to be the minimum of $\length{s}$ for each switch $s$ that meets $\Lambda(\Tau)$.
We define $\lambda(\Tau)$ to be the subset of branches of $\Lambda(\Tau)$ which minimise $\length{\cdot}$.
\end{definition}

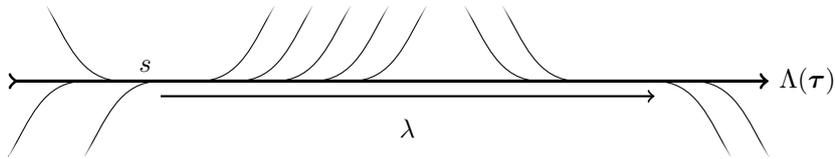
\begin{figure}[ht]
\centering
\input{tikz/long_path}
\caption{A switch $s$ and the train path $\lambda$ out of it following $\Lambda(\Tau)$ (shown in bold). Hence $\ell(s) = 8$.}
\label{Fig:switch_length}
\end{figure}

We note that if $\Tau$ is a clean pair of weighted train tracks then $\ell(\Tau) = \infty$ if and only if $\Lambda(\Tau)$ is a collection of compatibly combed cycles.
Furthermore, if $\ell(\Tau) < \infty$ then $0 \leq \length{\Tau} \leq \bee$.

\begin{lemma}
Suppose that $\Tau = (\sigma, \mu, \tau, \nu)$ is a clean pair of weighted train tracks.
If $\Tau$ is not crossing then $\Tau$ either contains a splittable shared branch or an incompatibly combed train cycle (with no switches). 
Furthermore, this is achieved by any shared branch $e$ which maximises $\complexity{\mu(e)} + \complexity{\nu(e)}$. \qed
\end{lemma}

\begin{definition}
Suppose that $\Tau = (\sigma, \mu, \tau, \nu)$ is a clean pair of weighted train tracks which is not crossing.
Choose $e$ a shared branch of $\Tau$ in $\lambda(\Tau)$.
Then:
\begin{enumerate}
\item If there is an incompatibly combed train cycle $c$ through $e$ then $\move(\Tau)$ denotes the result of separating along $c$,
\item If there is a compatibly combed train cycle $c$ through $e$ with rotation number $r \geq 1$ then $\move(\Tau)$ denotes the result of untwisting along $c$, and
\item Otherwise $\move(\Tau)$ denotes the result of splitting $e$, followed by splitting any isolated shared branches (if there are any). \qedhere
\end{enumerate}
\end{definition}

We note that $\move(\Tau)$ does not have any isolated shared branches, so it too is a clean pair of weighted train tracks.

\begin{remark}
Our combinatorial moves look at big sections of track whereas Dynnikov examines the two-neighbourhood of the largest branch \cite[Figures~22--33]{Dynnikov22}.
This is a fundamental difference between our work and his.
\end{remark}

\begin{proposition}
\label[proposition]{drop_complexity}
Suppose that $\Tau$ is a clean pair of weighted train tracks which is not crossing.
Then $\sharedsize{\move(\Tau)} \leq \sharedsize{\Tau}$.
Furthermore $\sharedsize{\move^k(\Tau)} < \sharedsize{\Tau}$ for some $k \leq 2 \bee^2$.
\end{proposition}

\begin{proof}
Suppose that $\Tau' \defeq \move(\Tau)$.
If $\tightness{\Tau'} < \tightness{\Tau}$ then $\Tau'$ has at most one more shared branch than $\Tau$; furthermore, this happens only in case 11c of \cref{conditional_splitting}.
Thus
\[ \sharedsize{\Tau'}_1 \leq \sharedsize{\Tau}_1 + \sharedsize{\Tau}_\infty \]
and so $\sharedsize{\Tau'} \leq \sharedsize{\Tau} - \sharedsize{\Tau}_\infty < \sharedsize{\Tau}$.
So in this case we are done. 

For the remainder of the proof we assume that $\tightness{\move^k(\Tau)} = \tightness{\Tau}$.
Consulting \cref{splitting,conditional_splitting} we find that for tightness-preserving splits there is a bijection between the shared branches of $\Tau$ and $\Tau'$.  
Furthermore, their bit-sizes are non-increasing. 
Thus $\sharedsize{\Tau'} \leq \sharedsize{\Tau}$.
This completes the proof of the first conclusion. 

We claim that
\[
(\sharedsize{\Tau'}_\infty, |\Lambda(\Tau')|, \ell(\Tau')) 
< 
(\sharedsize{\Tau}_\infty, |\Lambda(\Tau)|, \ell(\Tau))
\]
where tuples are compared lexicographically.
If $\ell(\Tau) = \infty$ then $\Lambda(\Tau)$ is a collection of compatibly combed cycles. 
In this case $\move(\Tau)$ applies an untwist; so $\sharedsize{\move(\Tau)}_\infty < \sharedsize{\Tau}_\infty$.
Otherwise, $|\Lambda(\Tau)|$ and $\ell(\Tau)$ are both bounded by linear functions of $\bee$, we have that $\sharedsize{\move^k(\Tau)}_\infty < \sharedsize{\Tau}_\infty$ for some $k \leq 2 \bee^2$.
Therefore, since tightness is constant, when $\sharedsize{\cdot}_\infty$ reduces so does $\sharedsize{\cdot}$.
Thus the second conclusion follows from the claim.

We now prove the claim. 
First note that if $\sharedsize{\Tau}_\infty = \sharedsize{\Tau'}_\infty$ then $\Lambda(\Tau') \subseteq \Lambda(\Tau)$.
Therefore $|\Lambda(\Tau')| \leq |\Lambda(\Tau)|$.

Second, if $\sharedsize{\Tau}_\infty = \sharedsize{\Tau'}_\infty$ and $|\Lambda(\Tau')| = |\Lambda(\Tau)|$ then the move performed must be a split and so $\ell(\Tau) < \infty$.
Let $e$ be the shared branch of $\Tau$ that is split and let $s$ and $s'$ be the switches incident to $e$.
The left (or right) small branch-ends of $s$ and $s'$ must both be in $\Lambda(\Tau)$, since otherwise splitting reduces $|\Lambda(\Tau)|$.
However this means that splitting $e$ reduces $\length{\Tau}$.
\end{proof}

Suppose that $\Tau$ is not crossing. 
Let $\move_{<}(\Tau)$ denote the result of applying $\move$ sufficiently many times to reduce $\sharedsize{\Tau}$.
By \cref{drop_complexity} at most $2 \bee^2$ applications of $\move$ are needed.
Let $\update(\Tau)$, $\update^k(\Tau)$, and $\update_{<}(\Tau)$ denote the update rule applied by $\move$, $\move^k$ and $\move_{<}$ to $\Tau$ respectively.

\begin{lemma}
Suppose that $\Tau$ is a clean pair of weighted train tracks.
Suppose that $U = \update^k(\Tau)$ for some $k$.
Then $\complexity{U} = O(\sharedsize{\Tau})$.
Thus, given such a $U$, we can compute $U(\Tau)$ in $O(m \log(m) + n)$ time where $n = \complexity{\Tau}$ and $m = \sharedsize{\Tau}$. \qed
\end{lemma}

This gives a (quadratic-time) algorithm to compute the intersection number of two curves as described in \cref{naive_intersection}.

\begin{algorithm}[ht!]
\caption{\textsc{NaiveIntersection}}
\label[algorithm]{naive_intersection}
\begin{algorithmic}[1]  
\Require{A clean pair of weighted train tracks $\Tau_1 = (\sigma, \mu, \tau, \nu)$.}
\Ensure{An integer.}
    \IfThen{$\sharedsize{\Tau_1} = 0$}{\Return $\pair{\mu, \nu}$} \Comment{Because crossing.}
    \State $U \gets \update_<(\Tau_1)$ \label{Line:naive_move} \Comment{An update rule that reduces shared size.}
    \State $\Tau_2 \gets U(\Tau_1)$ \label{Line:naive_apply} \Comment{$\Tau_2 = \move_{<}(\Tau_1)$.}
    \State \Return $\Call{NaiveIntersection}{\Tau_2}$ \label{Line:naive_recurse} \Comment{Recurse.}
\end{algorithmic}
\end{algorithm}

\begin{proposition}
Suppose that $\Tau_1 = (\sigma, \mu, \tau, \nu)$ is a clean pair of weighted train tracks.
Suppose that $\mu \in V(\sigma)$ and $\nu \in V(\tau)$ and that $\alpha \defeq C(\mu)$ and $\beta \defeq C(\nu)$.
\cref{naive_intersection} returns $\intersection(\alpha, \beta)$.
\end{proposition}

\begin{proof}
We proceed by induction on $\sharedsize{\Tau_1}$.
As the base case, when $\sharedsize{\Tau_1} = 0$ we have that $\Tau_1$ is crossing and so Line~\ref{Line:base} returns $\intersection(\alpha, \beta)$ by \cref{pair_is_intersection}. 
Otherwise, by the definition of $\update_{<}$ we have that $\sharedsize{\Tau_2} < \sharedsize{\Tau_1}$.
Therefore the recursive call to \textsc{NaiveIntersection} on Line~\ref{Line:naive_recurse} is well defined and returns $\intersection(\alpha, \beta)$ by induction.
\end{proof}

\begin{remark}
Furthermore, \cref{naive_intersection} runs in $O(n^2)$ time where $n = \complexity{\Tau_1}$. 
However to obtain this bound, care must be taken when performing untwisting since these require trial division.
\end{remark}

\begin{proposition}
Suppose that $\Tau = (\sigma, \mu, \tau, \nu)$ is a clean pair of weighted train tracks.
Suppose that $\Tau' = (\sigma', \mu', \tau', \nu') \defeq \move^k(\Tau)$ for some $k$.
Then 
\[ |\sigma' \pitchfork \tau'| - |\sigma \pitchfork \tau| \leq 13 \bee^3, \]
which, in particular, is independent of $k$. 
\end{proposition}

\begin{proof}
A new crossing point between $\sigma'$ and $\tau'$ is only created when a separation or split of type 2, 4, 5, 7, 8c, 10c or 11c is performed.
However, all of these moves except a split of type 7 reduce $\tightness{\Tau}$ by at least the number of crossing points that they introduce.
Thus the total number of crossing points introduced by separations or splits of type 2, 4, 5, 8c, 10c or 11c is at most $3 \bee$.

To bound the number of crossing points introduced by splits of type 7, suppose that $\tightness{\Tau'} = \tightness{\Tau}$.
Now if $|\sigma' \pitchfork \tau'| - |\sigma \pitchfork \tau| > 4 \bee^2$ then there a branch-end of $\sigma$ and a branch-end of $\tau$ which perform least five splits of type 7 together.
Without loss of generality, we may assume that the $\sigma$ branch-end is part of a left $\sigma$--switch and $\tau$ branch-end is part of a right $\tau$--switch.

Noting that we are only performing tightness-preserving moves, since this pair of branch-ends perform a split of type 7 together they must connect into the same smooth component of $\partial (S - n(\sigma \cap \tau))$.
Furthermore, since they perform at least two splits of type 7 together this component must actually be a circle.
Finally, since they perform at least five splits of type 7 
this component, when considered as a train cycle $c$ in $\sigma \cap \tau$, cannot meet any shared switches, right $\sigma$--switches, or left $\tau$--switches.
However this means that $c$ is an incompatibly combed train cycle.  
Thus, if any of the branches along it were ever considered for splitting then the move would separate the entire cycle instead.
This contradicts the assumption that $\tightness{\Tau'} = \tightness{\Tau}$.

Hence a sequence of tightness-preserving moves contains at most $4 \bee^2$ splits of type 7.
However any sequence of moves contains at most $3 \bee$ sub-sequences of tightness-preserving moves.
Therefore the number of crossing points introduced by splits of type 7 is at most $12 \bee^3$.

Hence $|\sigma' \pitchfork \tau'| - |\sigma \pitchfork \tau| \leq 12 \bee^3 + 3 \bee \leq 13 \bee^3$ as required.
\end{proof}

\begin{corollary}
Since we start with a clean pair of weighted standard train tracks and these have at most $\bee$ crossing points, every pair of train tracks that we encounter from now on will have at most $14 \bee^3 = O(1)$ crossing points. \qed
\end{corollary}

\section{Coarse intersection}

In this section, we show how to compute $\move^k(\Tau)$ in quasi-linear time.
This is inspired by the ``half GCD'' algorithm (HGCD-D) of M\"{o}ller~\cite{Moller08} and the similar observation that the amount of progress we make is actually proportional to the number of bits we need to look at.
In particular, \cref{exp_shorten} should be compared against \cite[Figure~6]{Moller08} and \cref{fast_intersection} should be compared against \cite[Figure~2]{Moller08}.

\subsection{Intervals}

To do this we approximate integers using \emph{(half-open) intervals} $[p, q)$.
There is a partial ordering on these interval where $[p, q) < [p', q')$ if and only if $q \leq p'$.
In addition to standard interval arithmetic, for intervals $J < I$ we also define $I \fdiv J$ to denote the largest integer $n$ such that $n J < I$.
For convenience, we also define $\{x\} \defeq [x, x+1)$.

\begin{definition}
Suppose that $I = [p, q)$ is an interval.
We define:
\begin{itemize}
\item its \emph{complexity} to be $\complexity{I} \defeq \complexity{p} + \complexity{q}$,
\item its \emph{complexity bound} to be $\complexitybnd{I} \defeq \complexity{q - 1}$, and
\item its \emph{uncertainty} to be $\uncertainty{I} \defeq \complexity{q - p}$. \qedhere
\end{itemize}
\end{definition}

In divide-and-conquer algorithms we must ``split the work''.
Inspired by the half-GCD algorithm we do this by taking the high order bits of a number, or rather of the intervals that approximate it.

\begin{definition}
Given an interval $I = [p, q)$ and a non-negative integer $k$ we define 
\[ 
\shift(I, k) \defeq [p \fdiv 2^k, q \cdiv 2^k). \qedhere
\]
\end{definition}

Shifting is an integral version of Thurston's notion of projectivisation of measured laminations; 
see~\cite[Page~vi]{FLP12} for a discussion. 

\subsection{Interval-weighted train tracks}

We recreate the machinery of \cref{Sub:train_track_pairs} using intervals.

\begin{definition}
Suppose that $\tau$ is a train track.
An \emph{interval-weighting} is a function $u \from B(\tau) \to \II \ZZ$ which assigns an interval to each branch.
We refer to the pair $(\tau, u)$ as an \emph{interval-weighted train track}.

We write a \emph{pair of interval-weighted train tracks} as $T = (\sigma, u, \tau, v)$.
Again we say that such a pair is tight / crossing / clean if the underlying pair of train tracks $(\sigma, \tau)$ are.
\end{definition}

Given a pair of weighted train tracks $\Tau = (\sigma, \mu, \tau, \nu)$, let $\{\Tau\} \defeq (\sigma, u, \tau, v)$ denote the pair of interval-weighted train tracks where 
\[ u(a) \defeq \{\mu(a)\} \inlineand v(b) \defeq \{\nu(b)\}. \]

The \emph{complexity} of a pair of interval-weighted train tracks $T = (\sigma, u, \tau, v)$ is 
\[ \complexity{T} \defeq \sum_{a \in B(\sigma)} \complexity{u(a)} + \sum_{b \in B(\tau)} \complexity{v(b)} \]
The \emph{$L^1$--shared size} and \emph{$L^\infty$--shared size} of $T$ are:
\[ 
\sharedsizebnd{T}_1 \defeq \sum_{\mathclap{e \in B(\sigma \cap \tau)}} \; \complexitybnd{u(e)} + \complexitybnd{v(e)}
\inlineand
\sharedsizebnd{T}_\infty \defeq \max_{\mathclap{e \in B(\sigma \cap \tau)}} \complexitybnd{u(e)} + \complexitybnd{v(e)}
\]
respectively.
Using these we define the \emph{shared size} of $T$ to be:
\[ 
\sharedsizebnd{T} \defeq \sharedsizebnd{T}_1 + 2 \cdot \tightness{T} \cdot \sharedsizebnd{T}_\infty. \qedhere
\]

Finally, define the \emph{certainty} of $T$ to be
\[ \omega(T) \defeq \complexitybnd{T}_\infty - \epsilon_\infty(T) \]
where
\begin{align*}
\complexitybnd{T}_\infty &\defeq \max_{e \in B(\sigma \cap \tau)} \max(\complexitybnd{u(e)}, \complexitybnd{v(e)}) \\
\epsilon_\infty(T) &\defeq \max_{e \in B(\sigma \cap \tau)} \max(\uncertainty{u(e)}, \uncertainty{v(e)}).
\end{align*}
This is a useful lower bound on the number of leading bits that are ``determined'' by their approximating intervals.

\subsection{Coarse moves}

Again there are coarse analogous of splitting, untwisting and separating for a tight pair of interval-weighted train tracks $T$.

However, when attempting to perform a conditional splitting a shared branch $e$ of $T$, it is possible that none of the conditions listed in \cref{conditional_splitting} are met.
This is due to the branch weight intervals being incomparable.
In which case we say that $e$ is \emph{not coarsely splittable}.
Some care is also needed when attempting to untwist since the rotation number $r$ that we use cannot be more than $2^{\omega(T)}$ to ensure that precision remains in our answer.

We define $\Lambda(T)$ to be the set of shared branches $e$ which maximise $\complexitybnd{u(e)} + \complexitybnd{v(e)}$ and $\lambda(T)$ to be the subset of these which minimise $\length{e}$.

\begin{definition}
Suppose that $T = (\sigma, u, \tau, v)$ is a clean pair of interval-weighted train tracks which is not crossing, that is $\sharedsizebnd{T} > 0$.
Choose $e$ a shared branch of $T$ in $\lambda(T)$.
\begin{enumerate}
\item If there is an incompatibly combed train cycle $c$ through $e$ then $\move(T)$ denotes the result of separating along $c$,
\item If there is a compatibly combed train cycle $c$ through $e$ with rotation number $r \geq 1$ then:
\begin{enumerate}
    \item if $\complexity{r} \leq \omega(T)$ then $\move(T)$ denotes the result of untwisting along $c$, otherwise
    \item $\move(T) \defeq \perp$, that is, $\move(T)$ is not defined,
\end{enumerate}
\item otherwise:
\begin{enumerate}
    \item if $e$ is coarsely splittable then $\move(T)$ denotes the result of splitting $e$, followed by splitting any isolated shared branches (if there are any), otherwise
    \item $\move(T) \defeq \perp$. \qedhere
\end{enumerate}
\end{enumerate}
\end{definition}

Similarly, let $\move_{<}(T)$ be the result of applying $\move$ until either $\sharedsizebnd{T}$ reduces or $\move(T)$ is not defined.
Again, following \cref{drop_complexity}, this occurs within $2 \bee^2$ applications of $\move$.
Let $\update(T)$, $\update^k(T)$ and $\update_{<}(T)$ denote the update rule applied by $\move$, $\move^k$ and $\move_{<}$ to $T$ respectively.
Note that these are just the identity update rule if $\move(T)$ is not defined.

\begin{lemma}
\label[lemma]{uncertainty_change}
Suppose that $T$ is a clean pair of interval-weighted train tracks.
Suppose that $U = \update^k(T)$ for some $k$.
Then 
\[
\epsilon_\infty(U(T)) - \epsilon_\infty(T) \leq \complexity{U}. \inlineQED
\]
\end{lemma}

Let $\cee \defeq 16 \bee^4$.

\begin{proposition}
\label[proposition]{update_interval_bound}
Suppose that $T$ is a clean pair of interval-weighted train tracks.
Suppose that $U = \update^k(T)$ for some $k$ and that $\sharedsizebnd{U(T)} < \sharedsizebnd{T}$.
Then $\complexity{U} \leq \cee \cdot (\sharedsizebnd{T} - \sharedsizebnd{U(T)})$.
\end{proposition}


\begin{proof}
Suppose that $U = (\sigma', \tau', A, B)$ and that $D \defeq \sharedsizebnd{T} - \sharedsizebnd{U(T)}$.  
We show that $A$ can be decomposed as a product of at most $8 \bee^2 D$ elementary matrices (and some number of projection matrices).
The $L_\infty$ bit-size of $A$ is then bounded above by $8 \bee^2 D$.
It follows that $\complexity{A} \leq 8 \bee^2 D \cdot \bee^2$.  
Since the same holds for $B$, we obtain the desired bound on $\complexity{U}$.

Consider
\[ 
U_i = (\sigma_i, \tau_i, A_i, B_i) \defeq \update(\move^{i-1}(T))
\]
Thus $U = U_k \circ \cdots \circ U_1$ and $A = A_k \cdot \cdots A_2 \cdot A_1$.
Each $U_i$ corresponds to either a split, a separation, or an untwisting.

First, suppose that $U_i$ corresponds to a split.
Then $A_i$ is:
\begin{itemize}
\item the identity matrix if $\sigma_i$ is not split by $U_i$,
\item a projection matrix if $U_i$ splits $\sigma_i$ centrally, or 
\item an elementary matrix if $U_i$ splits $\sigma_i$ non-centrally.
\end{itemize}
By \cref{drop_complexity} there are at most $2 \bee^2 (D + 1)$ such splits.
Therefore splits contribute at most $4 \bee^2 D$ elementary matrices to the decomposition of $A$.

Second, suppose that $U_i$ is an untwisting of order $p_i \neq 0$. 
Then, perhaps after conjugation by a permutation matrix, $A_i$ has the following form:
\[ 
\left(\begin{matrix}
\Id_x & -p_i \cdot 1_{x \times y} & 0 \\
0 & \Id_y & 0 \\
0 & 0 & \Id
\end{matrix}\right) 
\]
Multiplying by at most $2\bee$ elementary matrices transforms the above into $E_{k\ell}^{p_i}$, where $E_{k\ell}$ is an elementary matrix.  
Applying~\cite[Lemma~2.3]{Riley05} we may write $E_{k\ell}^{p_i}$ as a product of at most $27\complexity{p_i}$ elementary matrices,
and so $A_i$ as a product of at most $27 \complexity{p_i} + 2 \bee \leq (27 + 2 \bee) \complexity{p_i}$ elementary matrices.
Furthermore, such an untwist reduces $\sharedsizebnd{T}$ by at least $\complexity{p_i}$.
Therefore untwists contribute at most $(27 + 2 \bee) D \leq 4 \bee^2 D$ elementary matrices to the decomposition of $A$.

Finally, suppose that $U_i$ is a separation.
Then $A_i$ is the identity matrix.
Therefore separations contribute zero elementary matrices to the decomposition of $A$.
\end{proof}

Let $\dee \defeq \cee + 1$.

\begin{corollary}
\label[corollary]{omega_change}
Suppose that $T$ is a clean pair of interval-weighted train tracks.
Suppose that $U = \update^k(T)$ for some $k$ and that $\sharedsizebnd{U(T)} < \sharedsizebnd{T}$.
Then 
\[
\omega(T) - \omega(U(T)) \leq \dee \cdot (\sharedsizebnd{T} - \sharedsizebnd{U(T)}). \inlineQED
\]
\end{corollary}


\subsection{Coarsening}

Just like how we can build coarser intervals, we can build a coarser pairs of interval-weighted train tracks.

\begin{definition}
Suppose that $T = (\sigma, u, \tau, v)$ is a clean pair of interval-weighted train tracks.
Then $\trunc(T, k) \defeq (\sigma, u', \tau, v')$ is defined to be the clean pair of interval-weighted train tracks where 
\[
u'(a) \defeq \shift(u(a), d) \inlineand v'(b) \defeq \shift(v(b), d)
\]
where $d \defeq \complexitybnd{T}_\infty - k$.
\end{definition}

\begin{remark}
Suppose that $T$ is a clean pair of interval-weighted train tracks.
Suppose that $T' = \trunc(T, k)$ for some $k < \omega(T)$.
Then $\omega(T') = k$.
\end{remark}

\begin{proposition}
\label[proposition]{coarse_guide}
Suppose that $T$ is a clean pair of interval-weighted train tracks and $T' = \trunc(T, k)$ for some $k$.
If $\move(T')$ is defined then $\update(T') = \update(T)$. \qed
\end{proposition}
 
This means that if we determine the effect of $\move^k$ on a coarse approximation of $T$ then we can quickly determine $\move^k(T)$.

\begin{corollary}
Suppose that $T$ is a clean pair of interval-weighted train tracks and $T' = \trunc(T, k)$ for some $k$.
If $\move(T)$ is not defined then $\move(T')$ is also not defined. \qed
\end{corollary}

\subsection{Coarse shortening}

We now have the tools needed to give a fast algorithm to compute $\move^k(T)$ for a large value of $k$.
The overall strategy is to repeatedly:
\begin{enumerate}
\item Take a coarser approximation $T'$ of $T$.
\item Compute $\move^{k'}(T')$ for a moderate value of $k'$.
\item Use this information to quickly compute $\move^{k'}(T)$
\item Compute $\move_{<}(T)$ to reduce $\sharedsizebnd{T}$ by at least $1$.
\end{enumerate}
To obtain a subquadratic algorithm, we note that we can apply the same strategy recursively when we need to compute $\move^{k'}(T')$.

To make this explicit, we define $\eee \defeq 2 \dee$.

\begin{algorithm}[ht!]
\caption{\textsc{ExpShorten}}
\label[algorithm]{exp_shorten}
\begin{algorithmic}[1]  
\Require{A clean pair of interval-weighted train tracks $T_1$.}
\Ensure{An update rule.}
    \IfThen{$\omega(T_1) \leq 1$}{\Return $\update_{<}(T_1)$}
    \State $T'_1 \gets \trunc(T_1, \omega(T_1) \cdiv 2)$ \label{Line:approximate_1} \Comment{A coarser approximation of $T_1$.} 
    \State $U_1 \gets \Call{ExpShorten}{T'_1}$ \label{Line:recurse_1} \Comment{Recurse.} 
    \State $T_2 \gets U_1(T_1)$ \label{Line:update_1} \Comment{The same moves apply.}
    \State $U_2 \gets \update_{<}(T_2)$ \label{Line:drop_1}  \Comment{Drop $\geq 1$ bit.}
    \State $T_3 \gets U_2(T_2)$ \label{Line:apply_1} \Comment{$T_3 = \move_{<}(T_2)$.}
    \State \Comment{Now repeat.}
    \State $T'_3 \gets \trunc(T_3, \omega(T_1) \cdiv 2)$ \label{Line:approximate_2} \Comment{We use $\omega(T_1)$, \emph{not} $\omega(T_3)$.}
    \State $U_3 \gets \Call{ExpShorten}{T'_3}$ \label{Line:recurse_2}
    \State $T_4 \gets U_3(T_3)$ \label{Line:update_2}
    \State $U_4 \gets \update_{<}(T_4)$ \label{Line:drop_2}
    \State \textcolor{gray}{$T_5 \gets U_4(T_4)$} \label{Line:apply_2}  \Comment{Not needed since we never use $T_5$.}
    \State \Return $U_4 \circ U_3 \circ U_2 \circ U_1$ \label{Line:compose} \Comment{The overall update we have applied.}
\end{algorithmic}
\end{algorithm}

\begin{theorem}
\label[theorem]{exp_shorten_correctness}
Suppose that $T_1$ is a pair of interval-weighted train tracks.
\Cref{exp_shorten} returns an update rule $U$ such that:
\begin{enumerate}
\item \label[conclusion]{Conc:image} $U = \update^k(T_1)$ for some $k$, and either:
\item
\begin{enumerate}
    \item \label[conclusion]{Conc:progress} $\sharedsizebnd{T_1} - \sharedsizebnd{U(T_1)} \geq \omega(T_1) / \eee$, or
    \item \label[conclusion]{Conc:undefined} $\move(U(T_1))$ is not defined.
\end{enumerate}
\end{enumerate}
Furthermore, \cref{exp_shorten} runs in $O(m \log^2(m) + n \log(m))$ time where $n = \complexity{T_1}$ and $m = \sharedsizebnd{T_1}$.
\end{theorem}

\begin{proof}
We proceed by induction on $\omega(T_1)$.
When $\omega(T_1) \leq 1$ we have that $U = \update_{<}(T_1)$ and so the conclusions hold immediately.
We may assume that \cref{exp_shorten_correctness} is true for all subcalls since $\omega(T'_1), \omega(T'_3) < \omega(T_1)$.

Now by induction $U_1 = \update^{k_1}(T'_1)$ for some $k_1$ and so $U_1 = \update^{k_1}(T_1)$ by \cref{coarse_guide}.
Additionally $U_2 = \update_{<}(T_2) = \update^{k_2}(T_2)$.
The same argument shows that when the entire process is repeated we have that $U_3 = \update^{k_3}(T'_3) = \update^{k_3}(T_3)$ and $U_4 = \update_{<}(T_4) = \update^{k_4}(T_4)$.
Therefore $U = \update^{k}(T_1)$ where $k \defeq k_1 + k_2 + k_3 + k_4$, and so \cref{Conc:image} holds.

Now suppose that \cref{Conc:undefined} does not hold.
That is, $\move(U(T_1))$ is defined and so in particular all moves that we encounter within this call of the algorithm are defined.
Let $T'_2 \defeq U_1(T'_1)$ then:
\begin{itemize}
\item If $\move(T'_2)$ is defined then by induction $\sharedsizebnd{T'_1} - \sharedsizebnd{T'_2} \geq \omega(T'_1) / \eee$ and so 
\begin{align*}
\sharedsizebnd{T_1} - \sharedsizebnd{T_3} 
& \geq \sharedsizebnd{T_1} - \sharedsizebnd{T_2} \\
& = \sharedsizebnd{T'_1} - \sharedsizebnd{T'_2} \\
& \geq \omega(T'_1) / \eee \\
& \geq \frac{1}{2} \cdot \omega(T_1) / \eee.  
\end{align*}
\item If $\move(T'_2)$ is not defined then we consider the move applied by $\move(T_2)$:
\begin{itemize}
\item If $\tightness{\move(T_2)} < \tightness{T_2}$ then $\sharedsizebnd{\move(T_2)} \leq \frac{1}{2} \sharedsizebnd{T_2}$.
\item If $\move(T_2)$ performs an untwist with rotation number $r$ then $\complexity{r} \geq \omega(T'_2) \geq \omega(T_2) / 2$.
Therefore untwisting $T_2$ along $c$ eliminates at least $\omega(T_2) / 2$ bits.
\item If $\move(T_2)$ performs a split then the weights around the branch of $T_2$ that $\move$ splits must agree for the first $\omega(T'_2) \geq \omega(T_2) / 2$ bits since otherwise they would also be comparable in $T'_2$.
\end{itemize} 
In any case we have that $T_3 = \move(T_2)$ and
\begin{align*}
\sharedsizebnd{T_1} - \sharedsizebnd{T_3}
& = \sharedsizebnd{T_1} - \sharedsizebnd{\move(T_2)} \\
& \geq \frac{1}{2} \cdot \omega(T_1) / \eee. 
\end{align*}
\end{itemize}
In either case we conclude that $\sharedsizebnd{T_1} - \sharedsizebnd{T_3} \geq \frac{1}{2} \cdot \omega(T_1) / \eee$.

Now if $\sharedsizebnd{T_1} - \sharedsizebnd{T_3} \geq \omega(T_1) / \eee$ then the following moves can only increase this difference further and so \cref{Conc:progress} is guaranteed to hold.

Otherwise $1 \leq \sharedsizebnd{T_1} - \sharedsizebnd{T_3} \leq \omega(T_1) / \eee$ and so by \cref{omega_change} we have that
\begin{align*}
\omega(T_3)
& \geq \omega(T_1) - \dee \cdot \omega(T_1) / \eee \\
& \geq (1 - \dee / \eee) \cdot \omega(T_1) \\
& \geq \omega(T_1) / 2.
\end{align*}
Therefore $T'_3$ of Line~\ref{Line:approximate_2} is well defined and so this argument can be repeated for the second block.
That is,
\[ \sharedsizebnd{T_3} - \sharedsizebnd{T_5} \geq \frac{1}{2} \omega(T_1) / \eee \]
and so by the triangle inequality between $\sharedsizebnd{T_1}$, $\sharedsizebnd{T_3}$ and $\sharedsizebnd{T_5}$ we have that
\[ \sharedsizebnd{T_1} - \sharedsizebnd{T_5} 
\geq \omega(T_1) / \eee \]
and so \cref{Conc:progress} holds.

Finally, let $t_3(m, n)$ bound the time needed for $\Call{ExpShorten}{T_1}$ to run when $\complexity{T_1} \leq n$ and $\sharedsizebnd{T_1} \leq m$.
Then \cref{Line:approximate_1,,Line:approximate_2} complete in $O(n)$ time; \cref{Line:recurse_1,,Line:recurse_2} each complete in $t_3(m/2, n/2)$;
\cref{Line:update_1,,Line:update_2} complete in $O(m \log(m) + n)$ time;
\cref{Line:drop_1,,Line:drop_2} complete in $O(n)$ time; 
\cref{Line:apply_1,,Line:apply_2} complete in $O(m \log(m) + n)$ time;
and \cref{Line:compose} completes in $O(m \log(m))$ time.
Therefore
\[
t_3(m, n) \leq 2 t_3(m/2, n/2) + O(m \log(m) + n)
\]
and so $t_3(m, n) = O(m \log^2(m) + n \log(m))$.
\end{proof}

\begin{algorithm}[ht!]
\caption{\textsc{FastIntersection}}
\label[algorithm]{fast_intersection}
\begin{algorithmic}[1]
\Require{A clean pair of weighted train tracks $\Tau_1 = (\sigma, \mu, \tau, \nu)$.}
\Ensure{An integer.}
    \IfThen{$\sharedsize{\Tau_1} = 0$}{\Return $\pair{\mu, \nu}$} \label{Line:base} \Comment{If crossing.}
    \State $U \gets \Call{ExpShorten}{\{\Tau_1\}}$ \label{Line:shorten} \Comment{An update rule that makes a lot of progress.}
    \State $\Tau_2 \gets U(\Tau_1)$ \label{Line:apply} \Comment{The same moves apply.}
    \State $\Tau_3 \gets \move_{<}(\Tau_2)$
    \State \Return $\Call{FastIntersection}{\Tau_3}$ \label{Line:recurse} \Comment{Recurse.}
\end{algorithmic}
\end{algorithm}

We define $\fee \defeq 7 \bee \cdot \eee$.

\begin{theorem}
\label[theorem]{fast_intersection_correctness}
Suppose that $\Tau_1 = (\sigma, \mu, \tau, \nu)$ is a clean pair of weighted train tracks where $\mu \in V(\sigma)$ and $\nu \in V(\tau)$. 
Suppose that $\alpha \defeq C(\mu)$ and $\beta \defeq C(\nu)$.
\Cref{fast_intersection} returns $\intersection(\alpha, \beta)$.
Furthermore, \cref{fast_intersection} runs in $O((m + n) \log^2(m))$ time where $n = \complexity{\Tau_1}$ and $m = \sharedsize{\Tau_1}$.
\end{theorem}

\begin{proof}
We proceed by induction on $\sharedsize{\Tau_1}$.
As the base case, when $\sharedsize{\Tau_1} = 0$ we have that Line~\ref{Line:base} returns $\intersection(\alpha, \beta)$ by \cref{pair_is_intersection}.

Otherwise, let $T_1 \defeq \{\Tau_1\}$ and $T_2 \defeq U(T_1)$ and note that $T_2 = \{\Tau_2\}$ by \cref{coarse_guide}.
Now, by \cref{exp_shorten_correctness}, we have that $\Tau_2 = \move^k(\Tau_1)$ for some $k$ and either:
\begin{enumerate}
\item $\sharedsize{\Tau_1} - \sharedsize{\Tau_2} = \sharedsizebnd{T_1} - \sharedsizebnd{T_2} \geq \omega(T_1) / \eee \geq \sharedsize{\Tau_1} / \fee$, or

\item $\move(T_2)$ is not defined.
\end{enumerate}
In the former case
\[ \sharedsize{\Tau_2} \leq (1 - 1 / \fee) \cdot \sharedsize{\Tau_1}. \]
In the latter case, the same case analysis as in the proof of \cref{exp_shorten_correctness} shows that
\[ \sharedsize{\Tau_3} \leq (1 - 1 / \fee) \cdot \sharedsize{\Tau_2}. \]
Therefore, in either case we have that
\[ \sharedsize{\Tau_3} \leq (1 - 1 / \fee) \cdot \sharedsize{\Tau_1} \]
and so the recursive call to \textsc{FastIntersection} on Line~\ref{Line:recurse} is well defined and returns $\intersection(\alpha, \beta)$ by induction.

Finally, let $t_4(m, n)$ bound the time needed for $\Call{FastIntersection}{\Tau_1}$ to run when $\complexity{\Tau_1} \leq n$ and $\sharedsize{\Tau_1} \leq m$.
Note that $\complexity{T_1} = 2 \complexity{\Tau_1} = 2 n$ and $\sharedsizebnd{T_1} = \sharedsize{\Tau_1} = m$.
Therefore \cref{Line:shorten} runs in $O(m \log^2(m) + n \log(m))$ time by \cref{exp_shorten_correctness}; 
\cref{Line:apply} runs in $O(m \log^2(m) + n \log(m))$ time since $\complexity{U} \leq \cee \cdot \sharedsize{\Tau_1} \leq \cee \cdot m$ by \cref{update_interval_bound}; 
and \cref{Line:recurse} runs in $t_4((1 - 1 / \fee) \cdot m, n)$ time.

Therefore
\[ t_4(m, n) \leq t_4((1 - 1 / \fee) \cdot m, n) + O(m \log^2(m) + n \log(m)) \]
and so $t_4(m, n) = O(m \log^2(m) + n \log^2(m)) = O((m + n) \log^2(m))$ by the regularity condition of the master theorem.
\end{proof}

This proves \cref{GeometricIntersection}.

Additionally, starting \cref{fast_intersection} from $(\sigma, \mu, \sigma, \mu)$ while tracking the update rules applied gives an algorithm to shorten $(\sigma, \mu)$ in $O(n \log^2(n))$ time. 
This proves \cref{CurveShortening}.

\section{Remarks and questions}

\begin{remark}
To implement \cref{exp_shorten,fast_intersection} efficiently one should take into account that:
\begin{itemize}
\item 
The constants that appear throughout the paper were chosen to shorten the proofs; 
several of them can be reduced.
\item 
In \cref{exp_shorten} if $\sharedsizebnd{T_3} \leq \sharedsizebnd{T_1} - \omega(T_1) / \eee$ then we may immediately return $U_2 \circ U_1$ and so skip Lines~\ref{Line:approximate_2} -- \ref{Line:apply_2}.
\item 
In \cref{fast_intersection} it is likely more efficient to revert to \cref{naive_intersection} once $\sharedsize{\Tau_1}$ is small (below some uniform bound) instead of recursing all the way to the base case where $\sharedsize{\Tau_1} = 0$.
\item 
Branches of $T = (\sigma, u, \tau, v)$ which are disjoint from $\sigma \cap \tau$ are never modified within \cref{exp_shorten}.
Therefore if these weights are passed by reference then this algorithm can be improved to run in $O(m \log^2(m) + \log(n) \log(m))$ and so  \cref{fast_intersection} can be made to run in $O((m + \log(n)) \log^2(m))$ time.
\qedhere
\end{itemize}
\end{remark}

\begin{remark}
If we allows \emph{stops} -- vertices of valence one on $\bdy S$ -- in our train tracks, then we can represent integral lamination with arcs.
In this case, \cref{fast_intersection} generalises to compute $\intersection(\alpha, \beta)$ for any pair of integral laminations, including when $\alpha$ and / or $\beta$ are arcs, in $O(n \log^2(n))$ time where $n = \complexity{\Delta(\alpha)} + \complexity{\Delta(\beta)}$.
\end{remark}

\begin{remark}
Suppose that $G$ is a group and that $F$ is a finite generating set for $G$.
An \emph{exponent word}~\cite[Definition~3.1]{GurevichSchupp07} over $F$ has the following form. 
\[
w = f_k^{p_k} \cdots f_2^{p_2} \cdot f_1^{p_1}
\]
where $f_i \in F$ and $p_i \in \ZZ$ are integers. 
(These also called \emph{zipped words} in~\cite[page~183]{Dynnikov22}.)
The \emph{complexity} of such a $w$ is $\complexity{w} \defeq \sum_i \complexity{p_i}$.

We claim that for $d \geq 3$ the word problem for exponent words in $\GL(d, \ZZ)$ generated by elementary matrices can also be solved in $O(n \log^2(n))$ time where $n = \complexity{w}$.
We prove this by reducing to \cref{many_matrix_multiply}.
We first express each syllable $E_i^{p_i}$ as a product of $O(\log(p_i))$ elementary matrices~\cite[Lemma~2.3]{Riley05}.
The product is obtained by converting the given binary representation of $p_i$ into its \emph{Zeckendorf representation}: a sum of non-consecutive Fibonacci numbers.
Preprocessing each $p_i$ this way takes $O(\complexity{p_i} \log^2(\complexity{p_i}))$ time~\cite[Section~1.7.2]{BrentZimmermann10}.
Thus preprocessing the entire exponent word takes $O(n \log^2(n))$ time, proving the claim.

In similar fashion, we obtain a $O(n \cdot \poly(\log(n)))$ time solution to the word problem for exponent words in the mapping class group when the generators are Dehn twists by applying \cref{construct_twist_delta_coordinate} to the base cases (\cref{Line:base_0,Line:base_1}) of \cref{delta_coordinate}.

We end by noting that in both cases the computational complexity of the word problem for exponent words is very sensitive to the exact choice of generators.
\end{remark}


In work in progress, the first author and Webb, have given a polynomial-time algorithm for the conjugacy problem in the mapping class group.

\begin{question}
Is there an algorithm for the conjugacy problem in the mapping class group whose running time is subquadratic in $n$? 
\end{question}


Dynnikov gave a \emph{uniform} quadratic-time solution (in the RAM model) to the word problem in braid groups~\cite[page~215, Proposition~1.13]{DDRW08}.
That is, the constants do not depend on the number of strands.
Dylan Thurston~\cite[page~2]{Thurston08} points out that this technique generalises to all surfaces.
This leads us to ask the following.

\begin{question}
Is there an algorithm for the word problem in the mapping class group whose running time is subquadratic in $n$ and has constants independent of $S$?
\end{question}




\begin{appendices}

\section{Splitting}

In this appendix we catalogue all possible splits of $\Tau = (\sigma, \mu, \tau, \nu)$ up to rotation, reflection and interchanging $\sigma$ and $\tau$.
In each case we draw $\sigma$ in red, $\tau$ in blue and $\sigma \cap \tau$ in black.

\begin{longtable}[t]{cccc}
& Source & Target & Change in $\tightness{\Tau}$ \\
\midrule
\row{} & \splitsource{red}{red}{blue}{blue} & \splittarget{red}{red}{red}{blue}{blue}{blue}{transparent}{transparent} & -1 \\
\row{} & \splitsource{red}{blue}{red}{blue} & \splittarget{blue}{transparent}{red}{red}{transparent}{blue}{red}{blue} & -1 \\
\row{} & \splitsource{black}{blue}{red}{red} & \splittarget{blue}{blue}{black}{red}{red}{red}{red}{transparent} & -1 \\
\row{} & \splitsource{black}{red}{blue}{red} & \splittarget{red}{red}{black}{blue}{transparent}{red}{blue}{red} & -1 \\
\row{} & \splitsource{black}{red}{blue}{black} & \splittarget{red}{red}{black}{blue}{blue}{black}{blue}{red} & -2 \\
\row{} & \splitsource{black}{red}{black}{blue} & \splittarget{red}{red}{black}{black}{blue}{blue}{black}{transparent} & 0 \\
\row{} & \splitsource{black}{black}{red}{blue} &  \splittarget{black}{black}{black}{red}{transparent}{blue}{red}{blue} & 0 \\
\caption{Splitting at $e$.}
\label{splitting}
\end{longtable}

\begin{longtable}[t]{ccccc}
& Source & Target & Condition & Change in $\tightness{\Tau}$ \\
\midrule
\row{} & \splitsource{black}{red}{black}{red} & \splittarget{red}{red}{black}{black}{red}{red}{black}{transparent} & $\mu(a) > \mu(b)$ & 0 \\
& & \splittarget{red}{red}{black}{black}{red}{red}{blue}{transparent} & $\mu(a) = \mu(b)$ & -1 \\
& & \splittarget{red}{red}{black}{black}{red}{red}{blue}{red} & $\mu(a) < \mu(b)$ & -1 \\
\midrule
\row{} & \splitsource{black}{black}{red}{red} & 
\splittarget{black}{black}{black}{red}{red}{red}{red}{transparent} & $\mu(a) > \mu(b)$ & -1 \\
& & \splittarget{black}{black}{black}{red}{red}{red}{transparent}{transparent} & $\mu(a) = \mu(b)$ & -2 \\
& & \splittarget{black}{black}{black}{red}{red}{red}{transparent}{red} & $\mu(a) < \mu(b)$ & -1 \\
\midrule
\row{} & \splitsource{black}{black}{black}{red} & \splittarget{black}{black}{black}{black}{red}{red}{black}{transparent}& $\mu(a) > \mu(b)$ & 0 \\
& & \splittarget{black}{black}{black}{black}{red}{red}{blue}{transparent} & $\mu(a) = \mu(b)$ & -2 \\
& & \splittarget{black}{black}{black}{black}{red}{red}{blue}{red} & $\mu(a) < \mu(b)$ & -1 \\
\midrule
\row{} & \splitsource{black}{black}{black}{black} & \splittarget{black}{black}{black}{black}{black}{black}{black}{transparent} & $\mu(a) > \mu(b)$, $\nu(a) > \nu(b)$ & 0 \\
& & \splittarget{black}{black}{black}{black}{black}{black}{blue}{transparent} & $\mu(a) = \mu(b)$, $\nu(a) > \nu(b)$ & -3 \\
& &\splittarget{black}{black}{black}{black}{black}{black}{blue}{red} & $\mu(a) < \mu(b)$, $\nu(a) > \nu(b)$ & -1 \\
& & \splittarget{black}{black}{black}{black}{black}{black}{transparent}{transparent} & $\mu(a) = \mu(b)$, $\nu(a) = \nu(b)$ & -5 \\
& & \vdots & \vdots & \\
& & \multicolumn{3}{c}{And $5$ more cases up to symmetry in $\mu$ and $\nu$.} \\
\caption{Conditional splitting at $e$.}
\label{conditional_splitting}
\end{longtable}

\section{Cut disk tightness verification}
\label[appendix]{cut_disk_verification}

We encode a complementary disk $R$ of a pair of train tracks as a string by walking $\bdy R$ anti-clockwise recording the colours of the boundary segments that we encounter (\texttt{r}, \texttt{b} or \texttt{g} for red, blue or gray) and the types of transitions between them (\texttt{-}, \texttt{L} or \texttt{V} for smooth, corner or cusp).
We call an occurrence of \texttt{r-g-b} or \texttt{b-g-r} a \emph{colour change} in $\bdy R$.
Under this scheme the four cases of \cref{Fig:branch_cross_disk} are encoded as pairs of strings:
\begin{itemize}
\item[I:] \texttt{rLb-g-bL} and \texttt{rLb-g-bL}
\item[J:] \texttt{rLb-g-} and \texttt{rVg-bL}
\item[C:] \texttt{rVgV} and \texttt{r-g-}
\item[S:] \texttt{rVg-} and \texttt{rVg-}
\end{itemize}
We build all possible combinations of $R$ (and then $R'$) from these seeds by repeatedly replacing the first occurrence of \texttt{g} with either \texttt{g-rLb-g}, \texttt{g-bLr-g}, \texttt{gVg}, \texttt{g-r-g} or \texttt{g-b-g}.
We note that we can stop our exploration when either:
\begin{itemize}
\item an \texttt{r-g-r} or \texttt{b-g-b} is created, or
\item once $R$ (and then $R'$) matches one of the cases in \cref{Fig:high_index_regions} and $Q$ is not a monochromatic cusped bigon.
\end{itemize}
This was done via the following Python script:


\begin{lstlisting}
def index4(boundary):  # Return 4 * index.
    return 4 - 2 * boundary.count("V") - boundary.count("L")

def num_colour_changes(boundary):
    boundary += boundary[:3]  # So we can see colour changes that wrap over the end.
    return boundary.count("r-g-b") + boundary.count("b-g-r")

def is_bigon(boundary):
    return index4(boundary) == 0 and "L" not in boundary

def is_monochromatic(boundary):
    return "r" not in boundary or "b" not in boundary

def is_monochromatic_bigon(boundary):
    return is_bigon(boundary) and is_monochromatic(boundary)

def is_legal(boundary):
    return num_colour_changes(boundary) >= (1 if is_bigon(boundary) else index4(boundary))

def substitutes(boundary):
    i = boundary.find("g")
    for rule in ["g-rLb-", "g-bLr-", "gV", "g-r-", "g-b-"]:
        new = boundary[:i] + rule + boundary[i:]
        if "r-g-r" not in new and "b-g-b" not in new:
            yield new

def tree(left, right):
    whole = left[2:-2] + right[2:-2]
    if not is_legal(left) or is_monochromatic_bigon(whole):
        for new in substitutes(left):
            yield from tree(new, right)
    elif not is_legal(right) or is_monochromatic_bigon(whole):
        for new in substitutes(right):
            yield from tree(left, new)
    else:
        yield whole

SEEDS = [
    ("rLb-g-bL", "rLb-g-bL"),  # I
    ("rLb-g-", "rVg-bL"),  # J
    ("rVgV", "r-g-"),  # C
    ("rVg-", "rVg-"),  # S
]
for left, right in SEEDS:
    for reachable in tree(left, right):
        assert is_legal(reachable)
\end{lstlisting}

\section{Standard pairs tightness verification}
\label[appendix]{standard_pairs_verification}

Following the same notation as \cref{cut_disk_verification}, we verify that all combinations of annuli attached to pants result in legal regions.
This was done via the following Python script which rely on the same helper functions as in \cref{cut_disk_verification}:

\begin{lstlisting}
from itertools import product

ANNULI = ["gV", "g-b-g-r-", "g-r-g-b-", "g-bLr-", "g-rLb-"]
PANTS = [
    ("", "", "",),
    ("gVrV",), ("g-r-g-b-", "",),
    ("gVgV",), ("gV", "",),
    ("gVrV",), ("gVbV",), ("g-b-g-r-g-r-g-b-",),
]

for pants in PANTS:
    for annuli in product(ANNULI, repeat=len(pants)):
        boundary = "".join(piece for pair in zip(pants, annuli) for piece in pair)
        assert is_legal(boundary)
\end{lstlisting}

\section{Constants}

We summarise the constants of this paper here for convenience.

\begin{align*}
\bee & \defeq 6(3g - 3 + b) \\
\cee & \defeq 16 \bee^4 \\
\dee & \defeq \cee + 1 \\
\eee & \defeq 2 \dee \\
\fee & \defeq 7 \bee \cdot \eee
\end{align*}

\end{appendices}

\bibliographystyle{plain}
\bibliography{bibliography}

\end{document}

%% file: settings.tex
\usepackage{amsmath, amsthm, amssymb}

\usepackage{microtype}
\usepackage[hyphens]{url}

\usepackage{colonequals}
\usepackage{caption}
\usepackage[position=b]{subcaption}
\usepackage[section]{placeins}
\usepackage{mathtools}
\usepackage{appendix}
\usepackage{listofitems}
\usepackage[dvipsnames]{xcolor}

\definecolor{codegreen}{rgb}{0,0.6,0}
\definecolor{codepurple}{rgb}{0.58,0,0.82}

\usepackage{listings}
\makeatletter
\let\c@equation\c@subsection
\let\c@figure\c@equation
\let\c@table\c@equation
\makeatother

\numberwithin{equation}{section}
\numberwithin{figure}{section}
\numberwithin{table}{section}

\theoremstyle{plain}
\newtheorem{theorem}[equation]{Theorem}
\newtheorem{corollary}[equation]{Corollary}
\newtheorem{lemma}[equation]{Lemma}

\newtheorem{proposition}[equation]{Proposition}

\theoremstyle{definition}
\newtheorem{definition/}[equation]{Definition}
\newtheorem{example/}[equation]{Example}
\newtheorem{question/}[equation]{Question}
\newtheorem{algorithm/}[equation]{Algorithm}
\newtheorem*{keywords}{keywords}

\theoremstyle{remark}
\newtheorem{remark/}[equation]{Remark}

\newtheoremstyle{dotless}{}{}{}{}{\bfseries}{}{ }{}
\theoremstyle{dotless}
\newtheorem*{ccode}{Mathematics Subject Classification 2010:}

\newenvironment{definition}
  {%
   \pushQED{\qed}\begin{definition/}}
  {\popQED\end{definition/}}

\newenvironment{example}
  {%
   \pushQED{\qed}\begin{example/}}
  {\popQED\end{example/}}

\newenvironment{question}
  {%
   \pushQED{\qed}\begin{question/}}
  {\popQED\end{question/}}

\newenvironment{remark}
  {%
   \pushQED{\qed}\begin{remark/}}
  {\popQED\end{remark/}}

\usepackage[outline]{contour}
\contourlength{0.1em}
\newcommand{\outline}[1]{\contour*{white}{#1}}

\usepackage{tikz}
\usetikzlibrary{calc}
\usetikzlibrary{patterns}
\usetikzlibrary{fadings}
\usetikzlibrary{shapes.misc}
\usetikzlibrary{shadows.blur}
\usetikzlibrary{decorations.pathreplacing}
\usetikzlibrary{decorations.pathmorphing}
\usetikzlibrary{decorations.markings}
\usetikzlibrary{shapes.geometric}
\tikzset{dot/.style={draw,shape=circle,fill=black,scale=0.4}}
\tikzset{
every picture/.style={
  execute at end picture={
    \path (current bounding box.south west) +(-0.05, -0.05) (current bounding box.north east) +(0.05, 0.05);
    }
  }
}

\newcommand{\bdy}{\partial}
\newcommand{\from}{\colon}
\newcommand{\cross}{\times}

\newcommand{\defeq}{\colonequals}

\DeclareMathOperator{\MCG}{MCG}
\DeclareMathOperator{\Mod}{Mod}

\DeclareMathOperator{\GL}{GL}
\DeclareMathOperator{\intersection}{\iota}
\DeclareMathOperator{\move}{\textsc{Mov}}
\DeclareMathOperator{\update}{\textsc{UMov}}
\DeclareMathOperator{\ind}{index} 
\DeclareMathOperator{\shift}{shift} 
\DeclareMathOperator{\trunc}{trunc} 
\DeclareMathOperator{\DCoord}{\textsc{DeltaCoordinate}}
\DeclareMathOperator{\poly}{poly}

\newcommand{\tightness}[1]{\# #1}
\newcommand{\complexity}[1]{||#1||}
\newcommand{\complexitybnd}[1]{\lceil \! \lceil #1 \rceil \! \rceil}
\newcommand{\sharedsize}[1]{||#1||^\cap}
\newcommand{\sharedsizebnd}[1]{\lceil \! \lceil #1 \rceil \! \rceil^\cap}
\newcommand{\uncertainty}[1]{\epsilon( #1 )}
\newcommand{\pair}[1]{\langle #1 \rangle}
\newcommand{\length}[1]{\ell(#1)}

\newcommand{\closure}[1]{\overline{#1}}

\newcommand{\fdiv}{%
  \mathbin{%
    \mkern-6mu
    \begin{tikzpicture}[baseline = -0.5ex]
      \node[inner sep=0pt,outer sep=0pt,rotate=-110] at (0,0) {$\xrightarrow{}$};
    \end{tikzpicture}
    \mkern-3mu
  }%
}%
\newcommand{\cdiv}{%
  \mathbin{%
    \mkern-6mu
    \begin{tikzpicture}[baseline = -0.5ex]
      \node[inner sep=0pt,outer sep=0pt,rotate=-110] at (0,0) {$\xleftarrow{}$};
    \end{tikzpicture}
    \mkern-3mu
  }%
}%

\newcommand{\calC}{\mathcal{C}}
\newcommand{\Tau}{\boldsymbol{\tau}}
\newcommand{\Id}{\textnormal{Id}}
\newcommand{\II}{\mathbb{I}}
\newcommand{\ZZ}{\mathbb{Z}}

\newcommand{\bee}{\mathcal{B}}  
\newcommand{\cee}{\mathcal{C}}
\newcommand{\dee}{\mathcal{D}}
\newcommand{\eee}{\mathcal{E}}
\newcommand{\fee}{\mathcal{F}}

\newcommand{\inlineand}{\quad \textrm{and} \quad}
\newcommand{\inlineQED}{\pushQED{\qed} \qedhere \popQED}
\newcommand{\naive}{na\"{i}ve}

\newcommand{\blfootnote}[1]{%
  \begingroup
  \renewcommand\thefootnote{}\footnote{#1}%
  \addtocounter{footnote}{-1}%
  \endgroup
}

\usepackage[pdfborderstyle={},pdfborder={0 0 0},pagebackref,pdftex]{hyperref}

\renewcommand*{\backref}[1]{}
\renewcommand*{\backrefalt}[4]{
  \ifcase #1 %
   [No citations.]%
  \or
   [#2]%
  \else
   [#2]%
  \fi
}

\usepackage{algorithm}
\usepackage{algpseudocode}

\algnewcommand{\IfThen}[2]{\State \algorithmicif\ #1\ \algorithmicthen\ #2}

\usepackage[capitalise,noabbrev]{cleveref}

\crefname{conclusion}{Conclusion}{Conclusions}

\usepackage{booktabs}
\usepackage{longtable}
\newcounter{rown}
\newcommand{\row}{\stepcounter{rown}\arabic{rown}}

\usepackage[textsize=tiny]{todonotes}

%% file: tikz_commands.tex
\pgfdeclareradialshading{fadeoutdisk}{\pgfpointorigin}{
  color(0)=(pgftransparent!100);
  color(22)=(pgftransparent!100);
  color(24)=(pgftransparent!0);
  color(30)=(pgftransparent!0)%
}
\pgfdeclarefading{fade out disk}{\pgfuseshading{fadeoutdisk}}

\newcommand{\switch}[3]{
\begin{tikzpicture}[scale=1,very thick]
\draw [#1] (-1.41, 0) to (0, 0);
\draw [#2] (0, 0) to [in=240,out=0] (1, 1);
\draw [#3] (0, 0) to [in=120,out=0] (1, -1);

\fill[white,path fading=fade out disk] (0,0) circle (1.45);
\end{tikzpicture}
}

\newcommand{\splitsource}[4]{
\begin{tikzpicture}[scale=1.5,thick,baseline={([yshift=-.8ex]current bounding box.center)}]

\draw (0, 0) to node {\outline{$e$}} (0.5, 0);
\draw [#1] (0.5, 0) to [out=0,in=240] node [pos=0.6] {\outline{$a$}} (1, 0.5);
\draw [#2] (-0.5, 0.5) to [out=-60,in=180] node [pos=0.4] {\outline{$b$}} (0, 0);
\draw [#3] (-0.5, -0.5) to [out=60,in=180] node [pos=0.4]{\outline{$c$}} (0, 0);
\draw [#4] (0.5, 0) to [out=0,in=120] node [pos=0.6] {\outline{$d$}} (1, -0.5);

\end{tikzpicture}
}

\newcommand{\splittarget}[8]{
\begin{tikzpicture}[scale=1.5,thick,baseline={([yshift=-.8ex]current bounding box.center)}]

\draw [#1] (-0.5, 0.5) to [out=-60,in=180] (0, 0.2);
\draw [#2] (0, 0.2) to (0.5, 0.2);
\draw [#3] (0.5, 0.2) to [out=0,in=240] (1, 0.5);
\draw [#4] (-0.5, -0.5) to [out=60,in=180] (0, -0.2);
\draw [#5] (0, -0.2) to (0.5, -0.2);
\draw [#6] (0.5, -0.2) to [out=0,in=120] (1, -0.5);

\draw [#7] (0, -0.2) to [out=0,in=180] (0.5, 0.2);
\draw [#8] (0, 0.2) to [out=0,in=180] (0.5, -0.2);

\ifthenelse{\equal{#7}{red}}{\ifthenelse{\equal{#8}{blue}}{\node [dot] at (0.25,0) {};}}{}

\ifthenelse{\equal{#7}{blue}}{\ifthenelse{\equal{#8}{red}}{\node [dot] at (0.25,0) {};}}

\end{tikzpicture}
}

\newcommand{\tikzdisk}[2]{
\begin{tikzpicture}[scale=0.85,very thick]
\draw [transparent] (0,0) circle (1.5);  
\tikzdiskcore{#1}{#2}
\end{tikzpicture}
}

\newcommand{\tikzdisksplit}[6]{
\begin{tikzpicture}[scale=1,very thick]
\tikzdiskcore{#1}{#2}

\pgfmathsetmacro{\sa}{(360 * (#3-1) / \n) - 180 + (90 * #5)}
\pgfmathsetmacro{\ea}{(360 * (#4-1) / \n) - 180 + (90 * #6)}
\draw [red, looseness=1] (#3) to [out=\sa,in=\ea] (#4);

\ifthenelse{#5=0}{\node [dot] at (#3) {};}{}
\ifthenelse{#6=0}{\node [dot] at (#4) {};}{}
\end{tikzpicture}
}

\newcommand{\tikzdiskcore}[2]{

\setsepchar{,}
\readlist\angles{#1}
\readlist\colours{#2}

\pgfmathsetmacro{\n}{\listlen\angles[]}
\pgfmathsetmacro{\nn}{\n+1}
\pgfmathsetmacro{\nnn}{\n-1}

\tikzset{B/.style={blue}}
\tikzset{R/.style={red}}
\tikzset{G/.style={gray, dotted}}

\foreach \k [
count=\i from 1,
evaluate=\i as \ii using {ifthenelse(\i==\nn,1,\i)},
evaluate=\ii as \a using {\angles[\ii]},
evaluate=\k as \d using {360*\k / \n}
] in {0,1,...,\n} {
    \pgfmathsetmacro{\rr}{ifthenelse(\a==180,1,ifthenelse(\a==90,1.25,1.5)}
    \coordinate (\i) at (\d:\rr);
}

\path [draw=none, fill=none] (1)
\foreach \k 
[
count=\i from 1,
count=\j from 2,
evaluate=\j as \jj using {ifthenelse(\j==\nn,1,\j)},
evaluate=\i as \a using {\angles[\i]}, 
evaluate=\jj as \aa using {\angles[\jj]},
evaluate=\k as \d using {360*\k / \n}
] in {0,1,...,\nnn} {
    .. controls +({ifthenelse(\a==180,\d+90,ifthenelse(\a==90,\d+90+45,\d+90+90)}:0.25) and +({ifthenelse(\aa==180,\d-60,ifthenelse(\aa==90,\d-60-45,\d-60-90)}:0.25) .. (\j)
};

\foreach \k 
[
count=\i from 1, 
count=\j from 2,
evaluate=\j as \jj using {ifthenelse(\j==\nn,1,\j)},
evaluate=\i as \a using {\angles[\i]},
evaluate=\jj as \aa using {\angles[\jj]},
evaluate=\k as \d using {360*\k / \n}
] in {0,1,...,\nnn} {
    \pgfmathsetmacro{\xx}{ifthenelse(\a==180,\d+90,ifthenelse(\a==90,\d+90+45,\d+90+90)}
    \pgfmathsetmacro{\yy}{ifthenelse(\aa==180,\d-60,ifthenelse(\aa==90,\d-60-45,\d-60-90)}
    \ifthenelse{\equal{\colours[\i]}{R}}{\draw [R] (\i) to [out=\xx,in=\yy] (\j);}{}
    \ifthenelse{\equal{\colours[\i]}{G}}{\draw [G] (\i) to [out=\xx,in=\yy] (\j);}{}
    \ifthenelse{\equal{\colours[\i]}{B}}{\draw [B] (\i) to [out=\xx,in=\yy] (\j);}{}
}

\foreach \i
[
evaluate=\i as \a using {\angles[\i]},
] in {1,...,\n} {
    \ifthenelse{\a=90}{\node [dot] at (\i) {};}{}
}
}

\newcommand{\standardannulus}[4]
{
\begin{tikzpicture}[scale=1.5,very thick]

\draw [thin] (-1, 1.1) to [out=90, in=90, looseness=0.5] (1, 1.1);
\draw [thin, fill=none] (-1, 1.1) 
 to [out=270, in=90] (-1, -0.7) to [out=270, in=270, looseness=0.5] (1, -0.7)
 to [out=90, in=270] (1, 1.1) to [out=270, in=270, looseness=0.5] (-1, 1.1);
\draw [thin] (-1, -0.7) to [out=270, in=270, looseness=0.5] (1, -0.7);

\draw [thin] (-1, 1.1) to (-1, -0.7);
\draw [thin] (1, 1.1) to (1, -0.7);
\draw [thin, dotted] (-1, -0.7) to [out=90, in=90, looseness=0.5] (1, -0.7);

\draw [dotted, looseness=0.5] (-1,0.2) to [out=90,in=90] (1,0.2);
\draw (0, 0.8) to (0, 0.55);
\draw (0, -1) to (0, -0.75);

\draw (-1,0.2) to [out=270,in=165] (-0.75, 0);
\draw (1,0.2) to [out=270,in=15] (0.75, 0);

\draw [#1] (-0.75, 0) to [out=-15, in=270] (0,0.55); 
\draw [#1] (0.75, 0) to [out=195, in=90] (0,-0.75); 
\draw [#2] (-0.75, 0) to [out=-15, in=90] (0,-0.75); 
\draw [#2] (0.75, 0) to [out=195, in=270] (0,0.55); 

\draw [#3, looseness=0.75] (-0.75, 0) to [out=-15, in=195] (0.75, 0); 
\draw [#4] (0, -0.75) to (0, 0.55); 

\ifthenelse{\equal{#3}{red}}{\ifthenelse{\equal{#4}{blue}}{\node [dot] at (0,-0.0835) {};}}{}

\ifthenelse{\equal{#3}{blue}}{\ifthenelse{\equal{#4}{red}}{\node [dot] at (0,-0.0835) {};}}

\end{tikzpicture}
}

%% file: tikz/example_S_0_4.tex
\begin{tikzpicture}[scale=0.4, thick]
    \coordinate (NWL) at (-20, 13);
    \coordinate (NWR) at (-14, 13);
    \coordinate (NEL) at (-10, 13);
    \coordinate (NER) at (-4, 13);
    \coordinate (SWL) at (-20, 5);
    \coordinate (SWR) at (-14, 5);
    \coordinate (SEL) at (-10, 5);
    \coordinate (SER) at (-4, 5);
    \coordinate (N) at (-12, 11.25);
    \coordinate (S) at (-12, 6.75);
    \coordinate (W) at (-20, 9);
    \coordinate (E) at (-4, 9);
    
    \draw (NWL) to [out=90, in=90, looseness=0.75] (NWR);
    \draw (NEL) to [out=90, in=90, looseness=0.75] (NER);
    
    \draw (NWL) 
        to [out=270, in=270, looseness=0.75] (NWR)
        to [out=270, in=270, looseness=1.50] (NEL)
        to [out=270, in=270, looseness=0.75] (NER)
        to (SER)
        to [out=270, in=270, looseness=0.75] (SEL)
        to [out=90, in=90, looseness=1.50] (SWR)
        to [out=270, in=270, looseness=0.75] (SWL)
        to (NWL);

    \draw [dotted, in=90, out=90, looseness=0.75] (SEL) to (SER);
    \draw [dotted, in=90, out=90, looseness=0.75] (SWL) to (SWR);
    
    \draw [red, dotted, bend left=90, looseness=0.25] (W) to (E);
    \draw [blue, dotted, bend right=90] (N) to (S);
    \draw [Green, dotted] (-20, 10) [out=90, in=90, looseness=0.25] to (-4, 8);
    \draw [Green, dotted] (-11, 11.4) [out=45, in=90, looseness=0.5] to (-4, 11);
    \draw [Green, dotted] (-20, 8) [out=90, in=135, looseness=0.5] to (-11, 6.6);

    \draw [red, bend right=90, looseness=0.25] (W) to (E) node [right] {$\alpha$};
    \draw [blue, bend left=90] (N) node [above] {$\beta$} to (S);
    \draw [Green] (-11, 11.4) [out=245, in=270, looseness=0.5] to (-20, 10) node [left] {$\gamma$};
    \draw [Green] (-4, 11) [out=270, in=270, looseness=0.5] to (-20, 8);
    \draw [Green] (-11, 6.6) [out=-45, in=270, looseness=0.5] to (-4, 8);

\end{tikzpicture}

%% file: tikz/crossing.tex
\begin{tikzpicture}[scale=1,very thick]

\draw [red] (-1.41, 0) to (1.41, 0);
\draw [blue] (0, -1.41) to (0, 1.41);
\node [dot] at (0,0) {};

\fill[white,path fading=fade out disk] (0,0) circle (1.45);

\end{tikzpicture}

%% file: tikz/tangency.tex
\begin{tikzpicture}[scale=1,very thick]

\draw (-1.41, 0) to (1.41, 0);

\fill[white,path fading=fade out disk] (0,0) circle (1.45);

\end{tikzpicture}

%% file: tikz/standard_tri_tri.tex
\begin{tikzpicture}[scale=1, very thick]

\coordinate (A1) at (-1, 2);
\coordinate (B1) at (1, 2);

\begin{scope} [rotate=120]
    \coordinate (A2) at (-1, 2);
    \coordinate (B2) at (1, 2);
\end{scope}

\begin{scope} [rotate=240]
    \coordinate (A3) at (-1, 2);
    \coordinate (B3) at (1, 2);
\end{scope}

\draw [thin, fill=none] (A1) 
 to [out=270, in=30] (B2) to [out=30, in=30, looseness=0.5] (A2)
 to [out=30, in=150] (B3) to [out=150, in=150, looseness=0.5] (A3) 
 to [out=150, in=270] (B1) to [out=270, in=270, looseness=0.5] (A1);
\draw [thin] (A1) to [out=90, in=90, looseness=0.5] (B1);
\draw [thin] (A2) to [out=210, in=210, looseness=0.5] (B2);
\draw [thin] (A3) to [out=330, in=330, looseness=0.5] (B3);

\draw [thin] (A1) to [out=270, in=30] (B2);
\draw [thin] (A2) to [out=30, in=150] (B3);
\draw [thin] (A3) to [out=150, in=270] (B1);

\draw (90:1.7) to (90:0.9);
\draw (210:1.7) to (210:0.9);
\draw (330:1.7) to (330:0.9);
\draw (90:0.9) to [out=270, in=30] (210:0.9);
\draw (210:0.9) to [out=30, in=150] (330:0.9);
\draw (330:0.9) to [out=150, in=270] (90:0.9);

\end{tikzpicture}

%% file: tikz/standard_tri_loop.tex
\begin{tikzpicture}[scale=1, very thick]

\coordinate (A1) at (-1, 2);
\coordinate (B1) at (1, 2);

\begin{scope} [rotate=120]
    \coordinate (A2) at (-1, 2);
    \coordinate (B2) at (1, 2);
\end{scope}

\begin{scope} [rotate=240]
    \coordinate (A3) at (-1, 2);
    \coordinate (B3) at (1, 2);
\end{scope}

\draw [thin, fill=none] (A1) 
 to [out=270, in=30] (B2) to [out=30, in=30, looseness=0.5] (A2)
 to [out=30, in=150] (B3) to [out=150, in=150, looseness=0.5] (A3) 
 to [out=150, in=270] (B1) to [out=270, in=270, looseness=0.5] (A1);
\draw [thin] (A1) to [out=90, in=90, looseness=0.5] (B1);
\draw [thin] (A2) to [out=210, in=210, looseness=0.5] (B2);
\draw [thin] (A3) to [out=330, in=330, looseness=0.5] (B3);

\draw [dotted, looseness=0.5, red] (-1.0, 1.8) to [out=90,in=90] (1.0, 1.8);

\draw [red] (90:1.0) to [out=90,in=280] (1.0, 1.8);
\draw [red] (-1.0, 1.8) to [out=260,in=90] (90:1.0);

\draw [blue] (330:0.9) to [out=150, in=270] (90:0.9);

\draw (90:1.7) to (90:0.9);
\draw (210:1.7) to (210:0.9);
\draw (330:1.7) to (330:0.9);
\draw (90:0.9) to [out=270, in=30] (210:0.9);
\draw (210:0.9) to [out=30, in=150] (330:0.9);

\end{tikzpicture}

%% file: tikz/standard_loop_loop.tex
\begin{tikzpicture}[scale=1, very thick]

\coordinate (A1) at (-1, 2);
\coordinate (B1) at (1, 2);

\begin{scope} [rotate=120]
    \coordinate (A2) at (-1, 2);
    \coordinate (B2) at (1, 2);
\end{scope}

\begin{scope} [rotate=240]
    \coordinate (A3) at (-1, 2);
    \coordinate (B3) at (1, 2);
\end{scope}

\draw [thin, fill=none] (A1) 
 to [out=270, in=30] (B2) to [out=30, in=30, looseness=0.5] (A2)
 to [out=30, in=150] (B3) to [out=150, in=150, looseness=0.5] (A3) 
 to [out=150, in=270] (B1) to [out=270, in=270, looseness=0.5] (A1);
\draw [thin] (A1) to [out=90, in=90, looseness=0.5] (B1);
\draw [thin] (A2) to [out=210, in=210, looseness=0.5] (B2);
\draw [thin] (A3) to [out=330, in=330, looseness=0.5] (B3);

\draw [dotted, looseness=0.5] (-1.0, 1.8) to [out=90,in=90] (1.0, 1.8);

\draw (90:1.0) to [out=90,in=280] (1.0, 1.8);
\draw (-1.0, 1.8) to [out=260,in=90] (90:1.0);

\draw (90:1.7) to (90:0.9);
\draw (210:1.7) to (210:0.9);
\draw (330:1.7) to (330:0.9);
\draw (90:0.9) to [out=270, in=30] (210:0.9);
\draw (210:0.9) to [out=30, in=150] (330:0.9);

\end{tikzpicture}

%% file: tikz/standard_loop_loop_offset.tex
\begin{tikzpicture}[scale=1, very thick]

\coordinate (A1) at (-1, 2);
\coordinate (B1) at (1, 2);

\begin{scope} [rotate=120]
    \coordinate (A2) at (-1, 2);
    \coordinate (B2) at (1, 2);
\end{scope}

\begin{scope} [rotate=240]
    \coordinate (A3) at (-1, 2);
    \coordinate (B3) at (1, 2);
\end{scope}

\draw [thin, fill=none] (A1) 
 to [out=270, in=30] (B2) to [out=30, in=30, looseness=0.5] (A2)
 to [out=30, in=150] (B3) to [out=150, in=150, looseness=0.5] (A3) 
 to [out=150, in=270] (B1) to [out=270, in=270, looseness=0.5] (A1);
\draw [thin] (A1) to [out=90, in=90, looseness=0.5] (B1);
\draw [thin] (A2) to [out=210, in=210, looseness=0.5] (B2);
\draw [thin] (A3) to [out=330, in=330, looseness=0.5] (B3);

\draw [dotted, looseness=0.5, red] (-1.0, 1.8) to [out=90,in=90] (1.0, 1.8);
\begin{scope}[rotate=120]
    \draw [dotted, looseness=0.5, blue] (-1.0, 1.8) to [out=90,in=90] (1.0, 1.8);
\end{scope}

\draw [red] (90:1.0) to [out=90,in=280] (1.0, 1.8);
\draw [red] (-1.0, 1.8) to [out=260,in=90] (90:1.0);

\draw [red] (90:0.9) to [out=270, in=30] (210:0.9);

\begin{scope}[rotate=120]
    \draw [blue] (90:1.0) to [out=90,in=280] (1.0, 1.8);
    \draw [blue] (-1.0, 1.8) to [out=260,in=90] (90:1.0);
    
    \draw [blue] (210:0.9) to [out=30, in=150] (330:0.9);
\end{scope}

\draw (90:1.7) to (90:0.9);
\draw (210:1.7) to (210:0.9);
\draw (330:1.7) to (330:0.9);
\draw (210:0.9) to [out=30, in=150] (330:0.9);

\end{tikzpicture}

%% file: tikz/combed_compatible.tex
\begin{tikzpicture}[scale=2, very thick]


\foreach \x in {0.4,1.3} {\draw [red] (\x-0.3,0.5) to [out=-90,in=180] (\x,0);}
\foreach \x in {0.1,0.3} {\draw [red] (\x,0) to [out=0,in=100] (\x+0.3,-0.5);}

\foreach \x in {0.9,1.5} {\draw [blue] (\x-0.3,0.5) to [out=-90,in=180] (\x,0);}
\foreach \x in {1.0,2.0} {\draw [blue] (\x,0) to [out=0,in=100] (\x+0.3,-0.5);}

\foreach \x in {0.7,1.8} {\draw [black] (\x-0.3,0.5) to [out=-90,in=180] (\x,0);}
\foreach \x in {1.5} {\draw [black] (\x,0) to [out=0,in=100] (\x+0.3,-0.5);}

\draw [>->] (0,0) -- (2.5,0) node [right] {$c$};


\end{tikzpicture}

%% file: tikz/combed_incompatible.tex
\begin{tikzpicture}[scale=2, very thick]


\foreach \x in {0.9,1.1,1.4} {\draw [red] (\x-0.2,0.5) to [out=-90,in=180] (\x,0);}
\foreach \x in {0.4,1.7} {\draw [blue] (\x,0) to [out=0,in=270] (\x+0.1,0.5);}

\foreach \x in {1.5} {\draw [red] (\x,0) to [out=0,in=90] (\x+0.1,-0.5);}
\foreach \x in {0.3,1.2} {\draw [blue] (\x,0) to [out=180,in=90] (\x-0.1,-0.5);}

\draw [>->,very thick, black] (0,0) -- (2,0) node [right] {$c$};


\node [] at(0.6, 0.4) {\small $1$};
\node [] at(0.8, 0.4) {\small $2$};
\node [] at(1.5, 0.4) {\small $3$};
\node [] at(0.2, 0.4) {\small $4$};

\end{tikzpicture}

%% file: tikz/combed_separate.tex
\begin{tikzpicture}[scale=2, very thick]


\foreach \x in {0.9,1.1,1.4} {\draw [red] (\x-0.2,0.5) to [out=-90,in=180] (\x,0.1);}
\foreach \x in {0.4,1.7} {\draw [blue] (\x,-0.1) to [out=0,in=270] (\x+0.1,0.5);}

\foreach \x in {1.5} {\draw [red] (\x,0.1) to [out=0,in=90] (\x+0.1,-0.5);}
\foreach \x in {0.3,1.2} {\draw [blue] (\x,-0.1) to [out=180,in=90] (\x-0.1,-0.5);}

\draw [>->,very thick, red] (0,0.1) -- (2,0.1);
\draw [>->,very thick, blue] (0,-0.1) -- (2,-0.1);


\node [dot] at (0.54, 0.1) {};
\node [dot] at (1.84, 0.1) {};
\node [dot] at (1.64, -0.1) {};

\node [] at(0.6, 0.4) {\small $1$};
\node [] at(0.8, 0.4) {\small $2$};
\node [] at(1.5, 0.4) {\small $3$};
\node [] at(0.2, 0.4) {\small $4$};

\end{tikzpicture}

%% file: tikz/long_path.tex
\begin{tikzpicture}[scale=2,thick]


\foreach \x in {0.75,3.5,3.75} {\draw [thin] (\x,0) to [out=180,in=-60] (\x-0.5,0.5);}
\foreach \x in {0.5, 1} {\draw [thin] (\x,0) to [out=180,in=60] (\x-0.5,-0.5);}
\foreach \x in {1.25, 1.5,1.75,2,2.25} {\draw [thin] (\x,0) to [out=0,in=240] (\x+0.5,0.5);}
\foreach \x in {4.25,4.5} {\draw [thin] (\x,0) to [out=0,in=120] (\x+0.5,-0.5);}

\draw [>->, very thick] (0,0) -- (5,0) node [right] {$\Lambda(\Tau)$};

\draw [->] (1,-0.1) to node [below=5pt] {$\lambda$} (4.25,-0.1);

\node [above] at (0.9, 0) {$s$};

\fill [white, path fading=south] (0,0.51) rectangle (5, 0.4);
\fill [white, path fading=north] (0,-0.51) rectangle (5, -0.4);

\end{tikzpicture}

%% file: main.bbl
\begin{thebibliography}{10}

\bibitem{AHT06}
Ian Agol, Joel Hass, and William Thurston.
\newblock The computational complexity of knot genus and spanning area.
\newblock {\em Trans. Amer. Math. Soc.}, 358(9):3821--3850, 2006.

\bibitem{teruaki}
Kazushi Ahara.
\newblock teruaki (computer software).
\newblock \url{http://www.aharalab.sakura.ne.jp/teruaki.html}, 1997.

\bibitem{flipper}
Mark Bell.
\newblock flipper (computer software).
\newblock \url{pypi.python.org/pypi/flipper}, 2013.

\bibitem{curver}
Mark Bell.
\newblock curver (computer software).
\newblock \url{pypi.python.org/pypi/curver}, 2017.

\bibitem{BrentZimmermann10}
Richard~P. Brent and Paul Zimmermann.
\newblock {\em Modern computer arithmetic}.
\newblock Cambridge Monographs on Applied and Computational Mathematics. Cambridge University Press, 2010.

\bibitem{XTrain}
Peter Brinkmann.
\newblock Xtrain (computer software).
\newblock {\url{https://gitorious.org/xtrain}}, 2009.

\bibitem{Chillingworth69}
D.~R.~J. Chillingworth.
\newblock A finite set of generators for the homeotopy group of a non-orientable surface.
\newblock {\em Proc. Cambridge Philos. Soc.}, 65:409--430, 1969.

\bibitem{CLRS22}
T.H. Cormen, C.E. Leiserson, R.L. Rivest, and C.~Stein.
\newblock {\em Introduction to algorithms, fourth edition}.
\newblock MIT Press, 2022.

\bibitem{Dehn11}
M.~Dehn.
\newblock \"{U}ber unendliche diskontinuierliche {G}ruppen.
\newblock {\em Math. Ann.}, 71(1):116--144, 1911.

\bibitem{Dehn38}
M.~Dehn.
\newblock Die {G}ruppe der {A}bbildungsklassen.
\newblock {\em Acta Math.}, 69(1):135--206, 1938.
\newblock Das arithmetische Feld auf Fl\"{a}chen.

\bibitem{Dehn87}
Max Dehn.
\newblock {\em Papers on group theory and topology}.
\newblock Springer-Verlag, New York, 1987.
\newblock Translated from the German and with introductions and an appendix by John Stillwell, With an appendix by Otto Schreier.

\bibitem{DDRW08}
Patrick Dehornoy, Ivan Dynnikov, Dale Rolfsen, and Bert Wiest.
\newblock {\em Ordering braids}, volume 148 of {\em Mathematical Surveys and Monographs}.
\newblock American Mathematical Society, Providence, RI, 2008.

\bibitem{Dynnikov22}
Ivan Dynnikov.
\newblock Counting intersections of normal curves.
\newblock {\em Journal of Algebra}, 607:181--231, oct 2022.

\bibitem{EricksonNayyeri13}
Jeff Erickson and Amir Nayyeri.
\newblock {T}racing compressed curves in triangulated surfaces.
\newblock {\em Discrete Comput. Geom.}, 49(4):823--863, 2013.

\bibitem{Farb06}
Benson Farb.
\newblock {\em {P}roblems on mapping class groups and related topics}, volume~74.
\newblock American Mathematical Soc., 2006.

\bibitem{FarbMargalit12}
Benson Farb and Dan Margalit.
\newblock {\em {A} primer on mapping class groups}, volume~49 of {\em Princeton Mathematical Series}.
\newblock Princeton University Press, Princeton, NJ, 2012.

\bibitem{FLP12}
Albert Fathi, Fran{\c{c}}ois Laudenbach, and Valentin Po{\'e}naru.
\newblock {\em Thurston's work on surfaces}, volume~48 of {\em Mathematical Notes}.
\newblock Princeton University Press, Princeton, NJ, 2012.
\newblock Translated from the 1979 French original by Djun M. Kim and Dan Margalit.

\bibitem{FortnowHomer03}
Lance Fortnow and Steve Homer.
\newblock A short history of computational complexity.
\newblock {\em Bull. Eur. Assoc. Theor. Comput. Sci. EATCS}, 80:95--133, 2003.

\bibitem{GurevichSchupp07}
Yuri Gurevich and Paul Schupp.
\newblock Membership problem for the modular group.
\newblock {\em SIAM J. Comput.}, 37(2):425--459, 2007.

\bibitem{Trains}
Toby Hall.
\newblock Trains (computer software).
\newblock {\url{http://pcwww.liv.ac.uk/maths/tobyhall/software/}}, 2007.

\bibitem{Dynn}
Toby Hall.
\newblock Dynn (computer software).
\newblock {\url{http://pcwww.liv.ac.uk/maths/tobyhall/software/}}, 2018.

\bibitem{Hamidi-Tehrani00}
Hessam Hamidi-Tehrani.
\newblock On complexity of the word problem in braid groups and mapping class groups.
\newblock {\em Topology Appl.}, 105(3):237--259, 2000.

\bibitem{HarveyHoeven21}
David Harvey and Joris van~der Hoeven.
\newblock Integer multiplication in time {$O(n \log n)$}.
\newblock {\em Ann. of Math. (2)}, 193(2):563--617, 2021.

\bibitem{HopcroftUllman79}
John~E. Hopcroft and Jeffrey~D. Ullman.
\newblock {\em Introduction to automata theory, languages, and computation}.
\newblock Addison-Wesley Series in Computer Science. Addison-Wesley Publishing Co., Reading, MA, 1979.

\bibitem{Agol21}
Ian~Agol (https://mathoverflow.net/users/1345/ian agol).
\newblock Quantitative word problem for 3-manifold groups.
\newblock MathOverflow.
\newblock URL:https://mathoverflow.net/q/382652 (version: 2021-02-05).

\bibitem{Humphries79}
Stephen~P. Humphries.
\newblock Generators for the mapping class group.
\newblock In {\em Topology of low-dimensional manifolds ({P}roc. {S}econd {S}ussex {C}onf., {C}helwood {G}ate, 1977)}, volume 722 of {\em Lecture Notes in Math.}, pages 44--47. Springer, Berlin, 1979.

\bibitem{Lickorish64}
W.~B.~R. Lickorish.
\newblock A finite set of generators for the homeotopy group of a {$2$}-manifold.
\newblock {\em Proc. Cambridge Philos. Soc.}, 60:769--778, 1964.

\bibitem{BH}
W.~Menasco and J.~Ringland.
\newblock Bh (computer software).
\newblock {\url{http://copper.math.buffalo.edu/BH/index.html}}, 1999--2011.

\bibitem{Moller08}
Niels M\"{o}ller.
\newblock On {S}chönhage's algorithm and subquadratic integer {GCD} computation.
\newblock {\em {M}athematics of {C}omputation}, 77(261):589--607, 2008.

\bibitem{Mosher99}
Lee Mosher.
\newblock Mapping class groups are automatic.
\newblock {\em Annals of Mathematics}, 142:303--384, 1995.

\bibitem{Mosher03}
Lee Mosher.
\newblock Train track expansions of measured foliations, 2003.

\bibitem{Nielsen43}
Jakob Nielsen.
\newblock Abbildungsklassen endlicher {O}rdnung.
\newblock {\em Acta Math.}, 75:23--115, 1943.

\bibitem{OlshanskiiShpilrain24}
Alexander Olshanskii and Vladimir Shpilrain.
\newblock Linear average-case complexity of algorithmic problems in groups, 2024.

\bibitem{PennerHarer92}
R.~C. Penner and J.~L. Harer.
\newblock {\em Combinatorics of train tracks}, volume 125 of {\em Annals of Mathematics Studies}.
\newblock Princeton University Press, Princeton, NJ, 1992.

\bibitem{Penner82}
Robert~Clark Penner.
\newblock {\em A {COMPUTATION} {OF} {THE} {ACTION} {OF} {THE} {MAPPING} {CLASS} {GROUP} {ON} {ISOTOPY} {CLASSES} {OF} {CURVES} {AND} {ARCS} {IN} {SURFACES}}.
\newblock ProQuest LLC, Ann Arbor, MI, 1982.
\newblock Thesis (Ph.D.)--Massachusetts Institute of Technology.

\bibitem{Riley05}
T.~R. Riley.
\newblock Navigating in the {C}ayley graphs of {${\rm SL}_N(\mathbb{Z})$} and {${\rm SL}_N(\mathbb{F}_p)$}.
\newblock {\em Geom. Dedicata}, 113:215--229, 2005.

\bibitem{SSS08}
Marcus Schaefer, Eric Sedgwick, and Daniel Stefankovic.
\newblock Computing {D}ehn twists and geometric intersection numbers in polynomial time.
\newblock In {\em Proceedings of the 20th Annual Canadian Conference on Computational Geometry, Montr{\'{e}}al, Canada, August 13-15, 2008}, 2008.

\bibitem{Schleimer08}
Saul Schleimer.
\newblock Polynomial-time word problems.
\newblock {\em Comment. Math. Helv.}, 83(4):741--765, 2008.

\bibitem{Shpilrain24}
Vladimir Shpilrain.
\newblock Open problems in combinatorial group theory, 2024.

\bibitem{Takarajima00a}
Itaru Takarajima.
\newblock A combinatorial representation of curves using train tracks.
\newblock {\em Topology Appl.}, 106(2):169--198, 2000.

\bibitem{Takarajima00b}
Itaru Takarajima.
\newblock A combinatorial representation of {$\partial{\bf D}^2$} using train tracks.
\newblock {\em Topology Appl.}, 106(2):199--216, 2000.

\bibitem{Thurston08}
Dylan Thurston.
\newblock On geometric intersection of curves in surfaces.

\bibitem{ThurstonTravaux79}
W.~Thurston.
\newblock {\em Travaux de {T}hurston sur les surfaces}.
\newblock Asterisque-Societe Mathematique de France. Societe Mathematique de France, 1979.

\bibitem{Thurston80}
William Thurston.
\newblock Geometry and topology of three-manifolds, 1980.

\bibitem{Thurston82}
William~P. Thurston.
\newblock Three-dimensional manifolds, kleinian groups and hyperbolic geometry.
\newblock {\em Bulletin (New Series) of the American Mathematical Society}, 6(3):357 -- 381, 1982.

\bibitem{branched}
Alden Walker.
\newblock branched (computer software).
\newblock \url{https://github.com/aldenwalker/branched}, 2014.

\end{thebibliography}
